\pdfoutput=1
\documentclass[11pt]{amsart}
\usepackage{amsmath,amssymb,amsthm}
\usepackage{booktabs}
\usepackage{adjustbox}
\usepackage[margin=1.2in]{geometry}
\usepackage[colorlinks=true,linkcolor=blue,citecolor=blue,urlcolor=blue]{hyperref}
\hypersetup{pdfauthor={Bernd Johannes Wuebben},
  pdftitle={Slope stability of tangent bundles of smooth toric Fano varieties}}

\newtheorem{theorem}{Theorem}[section]
\newtheorem{proposition}[theorem]{Proposition}
\newtheorem{corollary}[theorem]{Corollary}
\newtheorem{lemma}[theorem]{Lemma}

\theoremstyle{definition}
\newtheorem{definition}[theorem]{Definition}
\newtheorem{example}[theorem]{Example}
\theoremstyle{remark}
\newtheorem{remark}[theorem]{Remark}

\newcommand{\PP}{\mathbb{P}}
\newcommand{\CC}{\mathbb{C}}
\newcommand{\ZZ}{\mathbb{Z}}
\newcommand{\QQ}{\mathbb{Q}}
\newcommand{\dP}{\mathrm{dP}}
\newcommand{\Vol}{\operatorname{Vol}}

\title[Stability of tangent bundles of smooth toric Fanos]{Slope stability
of tangent bundles\\ of smooth toric Fano varieties}

\author{Bernd Johannes Wuebben}
\subjclass[2020]{14M25 (primary); 14J45, 14J60, 32Q20, 53C07 (secondary)}
\keywords{Toric Fano varieties, tangent bundle, slope stability,
Hermitian--Einstein metrics, K\"ahler--Einstein metrics, Klyachko
filtrations, reflexive polytopes}
\date{Original version: August 13, 2026; revised: September 20, 2026}

\begin{document}

\begin{abstract}
We construct smooth toric Fano $n$-folds of Picard number $n+2$ whose
tangent bundles are slope-stable with respect to the anticanonical
polarization, for every $n\ge4$. The construction gives toric fibrations
with fibre the del Pezzo surface of degree six over products of projective
lines, parametrized by multisets of roots of $A_2$. Every nonempty multiset
of nonzero roots with vanishing sum yields a stable tangent bundle. Among
these multisets, the resulting variety is K\"ahler--Einstein if and only
if the multiset is invariant under negation or under the order-three
rotation of the root hexagon. We also construct stable,
non-K\"ahler--Einstein examples with nonzero twist sum in every even
dimension at least four. The proofs reduce slope inequalities and
barycenter computations to integrals over the moment hexagon.

For a smooth toric Fano variety with strictly semistable tangent bundle,
we prove that polystability is equivalent to decomposition as a nontrivial
product of smooth toric Fano varieties with anticanonically stable tangent
bundles. Exact evaluation of Klyachko's criterion extends the stability
classification to dimensions five and six, recovering the classifications
of Steffens and Reynolds in dimensions three and four. Together with the
product criterion, this determines polystability for all $8{,}630$
varieties in dimensions three through six, and hence the existence of
Hermitian--Einstein metrics on their tangent bundles with respect to an
anticanonical K\"ahler form.
\end{abstract}

\maketitle

\section{Introduction}

Let $X$ be a Fano manifold of dimension $n$, polarized by $-K_X$. Two
conditions of an Einstein type compete on $X$: the existence of a
K\"ahler--Einstein metric on the manifold, and the existence of a
Hermitian--Einstein metric on the tangent bundle with respect to a K\"ahler
form in $c_1(X)$. By the Kobayashi--Hitchin correspondence of
Donaldson and Uhlenbeck--Yau \cite{Donaldson1985,UhlenbeckYau1986}, the
second condition is equivalent to the polystability of $T_X$ in the sense
of Mumford--Takemoto slopes; and if $X$ is K\"ahler--Einstein, the
K\"ahler--Einstein metric is itself Hermitian--Einstein on $T_X$, so
K\"ahler--Einstein implies polystable
\cite{Kobayashi1987,Lubke1983}. The converse fails: Fahlaoui
\cite{Fahlaoui1989} showed that the blow-up of $\PP^2$ in two points has
stable tangent bundle, although by Matsushima's theorem
\cite{Matsushima1957} it admits no K\"ahler--Einstein metric. The
relationship between the two conditions beyond this single implication has
remained largely unexplored, for a simple reason: outside of homogeneous
and low-dimensional examples, deciding the stability of a tangent bundle
is hard.

Toric varieties are the exception. Klyachko's classification of
equivariant vector bundles \cite{Klyachko1990}, in the form developed by
Perling and Kool for reflexive sheaves \cite{Perling2004,Kool2011},
reduces slope inequalities for equivariant sheaves on a smooth toric
variety to finitely many exact comparisons of lattice quantities.
Hering, Nill and S\"u\ss{} \cite{HNS2022} distilled this into a clean
combinatorial criterion for tangent bundles
(Proposition~\ref{prop:criterion} below) and used it to classify
stability of $T_X$ over the full ample cone for all smooth toric surfaces
and all smooth toric varieties of Picard rank $2$. On the
K\"ahler--Einstein side, the toric case is equally decidable: by
Mabuchi and Wang--Zhu \cite{Mabuchi1987,WangZhu2004}, a smooth toric Fano
is K\"ahler--Einstein if and only if the barycenter of its moment
polytope vanishes.

Complete classification lists are available in all dimensions of our census:
$5$ surfaces, $18$ three-folds \cite{WatanabeWatanabe1982,Batyrev1999}, $124$
four-folds \cite{Batyrev1999,Sato2000}, and, by \O bro's algorithm
\cite{Obro2007}, $866$ five-folds and $7622$ six-folds. Thus both sides of
the comparison are, in principle, finite computations across the entire
landscape. In practice the literature covers: dimension $2$
\cite{Fahlaoui1989,HNS2022}; dimension $3$, where Steffens
\cite{Steffens1996} classified $(-K)$-stability of $T_X$ for \emph{all}
smooth Fano three-folds; and dimension $4$ for Picard rank
$\rho(X) \le 2$ (Biswas--Dey--Gen\c{c}--Poddar \cite{BDGP2021}) and
$\rho \le 3$ (Dasgupta--Dey--Khan \cite{DDK}). Reynolds's 2023 thesis
\cite[Appendix~A]{Reynolds2023} gives the complete anticanonical
stability classification of all $124$ smooth toric Fano four-folds, in
every Picard rank, and the full Harder--Narasimhan filtrations of the
$74$ unstable tangent bundles.

This paper constructs smooth toric Fano varieties of Picard rank $n+2$
with stable tangent bundles in every dimension $n\ge4$ and characterizes
polystability of strictly semistable tangent bundles by product
decomposition. We also extend the stability classification to dimensions
five and six, determine polystability in dimensions three through six,
and compare the resulting classes with the K\"ahler--Einstein locus.
Our main results are the following.

\begin{theorem}[Classification; Theorem~\ref{thm:main}]
Of the smooth toric Fano varieties of dimension $n = 3, 4, 5, 6$, exactly
$4$, $26$, $109$, $636$ have $(-K)$-stable tangent bundle; exactly
$7$, $24$, $82$, $366$ are strictly semistable; and exactly $7$,
$74$, $675$, $6620$ are unstable.
\end{theorem}

The dimension-three and dimension-four classifications are due to
Steffens \cite{Steffens1996} and Reynolds \cite[Appendix~A]{Reynolds2023},
respectively; the counts in dimensions five and six extend these results.

\begin{theorem}[Hermitian--Einstein classification; Theorem~\ref{thm:HE}]
Exactly $8$ of the $18$ smooth toric Fano three-folds, $40$ of the $124$
four-folds, $161$ of the $866$ five-folds, and $895$ of the $7622$
six-folds admit a Hermitian--Einstein metric on their tangent bundle
with respect to a K\"ahler form representing $c_1(X)$.
\end{theorem}

The passage from the classification to the census requires deciding, for
each strictly semistable $T_X$, whether it is polystable. For tangent
bundles of toric varieties this is again purely combinatorial:
polystability of an equivariant reflexive sheaf can always be
exhibited by an equivariant splitting (Lemma~\ref{lem:equivpoly}, proved
for arbitrary polystable equivariant reflexive sheaves and perhaps of
independent use), and
the two-step structure of the Klyachko filtrations then implies that
equivariant direct-sum decompositions of $T_X$ correspond exactly to
partitions of the ray generators into groups spanning complementary
subspaces (Proposition~\ref{prop:split}); polystability becomes a finite
search over such partitions. More canonically, the connected components of
the ray matroid give the finest splitting and force a product decomposition of
the variety. Consequently a strictly semistable tangent bundle is polystable
exactly when the variety is a nontrivial product of smooth toric Fanos with
stable tangent bundles (Corollary~\ref{cor:matroidsplit}). The refinement is
geometrically meaningful. For instance, the
blow-up of $\PP^3$ along a line is strictly semistable but not polystable
(the three-dimensional analogue of the Hirzebruch surface
$\mathbb{F}_1$), and thus carries no Hermitian--Einstein metric despite
being semistable; while $\dP_2 \times \PP^1$ is polystable, giving a
Hermitian--Einstein tangent bundle on a manifold with no
K\"ahler--Einstein metric.

The data exhibit three phenomena.

\emph{First, instability is the dominant condition, and increasingly so
with dimension:} the unstable fraction grows from $7/18$ in dimension
three to $74/124$ in dimension four to $675/866$ in dimension five. The
Hermitian--Einstein locus thins correspondingly.

\emph{Second, K\"ahler--Einstein is a much rarer condition than a
Hermitian--Einstein tangent bundle, and the gap widens:} in dimension
five, $109$ varieties have stable $T_X$ but only $7$ of them are
K\"ahler--Einstein. Fahlaoui's surface example is not an exception; it is
the first instance of a widespread pattern in the census. In the other direction our data
confirm, in every dimension computed, that all K\"ahler--Einstein
varieties have polystable tangent bundle, as the Kobayashi--Hitchin
correspondence requires; this includes its combinatorial form, that a
vanishing barycenter forces the subspace concentration
inequalities of \cite{HNS2022,HenkLinke2014}.

\emph{Third, at high Picard rank stability collapses to a single,
structured example.} Among the $42$ five-folds with $\rho \ge 7$ exactly
one is stable. It is K\"ahler--Einstein, and it has a transparent
geometric description: a toric $\dP_3$-fibration over $(\PP^1)^3$ in
which the three base directions are twisted by the three roots
$\alpha, \beta, -(\alpha+\beta)$ of $A_2$. These three twists form an orbit
of the order-three rotation of the root hexagon, a symmetry that forces the
barycenter to vanish (hence the K\"ahler--Einstein property), while the
twist destroys the direct-sum decomposition that makes the untwisted product
$\dP_3 \times (\PP^1)^3$ merely polystable: it turns every slope equality
into a strict inequality. We define this root-twist construction for every
number of base factors (Section~\ref{sec:roottwist}). Exact computation first
verified zero-sum stability through dimension eight; Theorem~\ref{thm:zerosum}
proves it in every dimension. Theorem~\ref{thm:roottwist-ke} further proves
that such a variety is K\"ahler--Einstein exactly when its twist multiset is
symmetric under negation or the order-three rotation. A
vanishing sum is not necessary for stability: we prove that the unbalanced twists
$((\alpha,-\alpha)^{\times r},\alpha,\beta)$ are stable and
non-K\"ahler--Einstein for every $r\ge0$. Together with the two symmetric
families, whose twists form a repeated $\pm\alpha$ pair or a repeated rotation
orbit, this gives three particularly explicit infinite families. The proofs
rest on an exact criterion (Proposition~\ref{prop:reduction}):
every slope inequality of the family is an inequality between integrals
over the two-dimensional moment hexagon of $\dP_3$, and stability is
equivalent to the negativity of one such integral per twist together
with a strict triangle inequality among three line shares, one attached
to each root line of the hexagon. In the general zero-sum case these
inequalities follow from centered-box convolution estimates.

\begin{theorem}[Zero-sum root twists; Theorems~\ref{thm:zerosum}
and~\ref{thm:roottwist-ke}]
Let $t_1,\dots,t_k$ be nonzero roots of $A_2$ with
$\sum_i t_i=0$. Then $T_{X(t_\bullet)}$ is $(-K)$-stable. Consequently,
for every $n\ge4$ there is a smooth toric Fano $n$-fold of Picard rank
$n+2$ with stable tangent bundle. Moreover, $X(t_\bullet)$ is
K\"ahler--Einstein if and only if the twist multiset is invariant under
negation or under the order-three rotation of the root hexagon.
\end{theorem}

\begin{theorem}[Three infinite families; Theorem~\ref{thm:families}]
For every $r \ge 1$, $m \ge 1$, and $s \ge 0$, the root-twist varieties
$X((\alpha,-\alpha)^{\times r})$, of dimension $n = 2r + 2$, and
$X((\alpha,\beta,-(\alpha{+}\beta))^{\times m})$, of dimension
$n = 3m + 2$, have $(-K)$-stable tangent bundle and are
K\"ahler--Einstein. The varieties
$X((\alpha,-\alpha)^{\times s},\alpha,\beta)$, of dimension $n=2s+4$,
also have stable tangent bundle but are not K\"ahler--Einstein. All three
families are smooth toric Fanos of Picard rank $n+2$. In particular, stable
K\"ahler--Einstein examples of rank $n+2$ exist in every even dimension
$n\ge4$ and every dimension $n\equiv2\pmod3$, $n\ge5$, while stable
non-K\"ahler--Einstein examples of that rank exist in every even dimension
$n\ge4$.
\end{theorem}

Finally, for every unstable variety the computation records the canonical
first Harder--Narasimhan subspace. These subspaces have structure: in $7/7$,
$68/74$, $558/675$ and $4877/6620$ of the unstable cases in dimensions
$3$, $4$, $5$, $6$, they contain an antipodal pair of rays. Such a pair is
the combinatorial signature of a candidate relative tangent direction for an
equivariant $\PP^1$-fibration, whose existence requires a separate check
of the fan. The presence of such a pair does not imply that a Mori fibration
has destabilizing relative tangent sheaf: Reynolds's four-dimensional
counterexample to Peternell's conjecture \cite[Theorem~2.2.6]{Reynolds2023}
already contains such a pair (Section~\ref{sec:antipodal}).

\subsection*{Related and concurrent work}
The combinatorial machinery of equivariant reflexive sheaves on toric
varieties has developed rapidly alongside this work: Napame and Tipler
study slope stability of reflexive pullbacks along toric fibrations and
produce stable perturbations of semistable tangent sheaves by varying
the polarization or blowing up \cite{NapameTipler2024}; Clarke, Napame
and
Tipler analyse how slope (poly)stability of toric sheaves behaves under
toric flips \cite{ClarkeNapameTipler}; Tipler studies the Chern classes
of stable toric sheaves \cite{Tipler2025}; and Delloque, Napame, Scarpa
and Tipler develop polynomial stability conditions with equivariant
positivity on $T$-varieties \cite{DNST2025}. On the Fano side, Kanemitsu
\cite{Kanemitsu2021} determined the stability of the tangent bundles of
Pasquier's two-orbit varieties, disproving the long-standing expectation
that a Fano manifold of Picard number one has stable tangent bundle.
Reynolds \cite[Theorem~2.2.6]{Reynolds2023} gives a toric counterexample
to Peternell's conjecture on destabilizing relative tangent sheaves
\cite[Conjecture~3.21]{Peternell2001} in dimension four and Picard rank
four, thus in smaller dimension but higher Picard rank than Kanemitsu's
examples. Concurrently with the present paper, Chen and Lai \cite{ChenLai2026}
characterize the maximal destabilizing sheaf of the tangent sheaf of a
Fano variety through the foliated minimal model program, classifying the
$(-K)$-slope-unstable weak del Pezzo surfaces and exhibiting a singular
del Pezzo surface that carries a weak K\"ahler--Einstein metric yet has
slope-unstable tangent sheaf. In the singular setting, therefore, a weak
K\"ahler--Einstein metric need not force slope-semistability of $T_X$.

\subsection*{Reliability}
Because the mathematical content of a classification of this kind is
concentrated in the correctness of the computation, we describe the
validation in detail (Section~\ref{sec:harness}). In brief: all
arithmetic is exact (integer and rational); every variety passes internal
consistency identities; and the implementation reproduces, without
discrepancy, the following prior classifications (Fahlaoui's del Pezzo
table; the
toric slice of Steffens's classification of Fano three-folds; the
$38$ four-folds of Picard rank $\le 3$ classified by
\cite{BDGP2021,DDK}, matched variety-by-variety via exact
$GL(4,\ZZ)$-equivalence of fans; Reynolds's complete four-fold
classification \cite[Appendix~A]{Reynolds2023}, including the rank and
degree of the first Harder--Narasimhan term in all $74$ unstable cases;
and the K\"ahler--Einstein tables of
\cite{Nakagawa1993,Nakagawa1994,NSY2023}). All analysis code, the input polytope
data, and the per-variety result files are publicly available at
\url{https://github.com/bwuebben/toric-tangent-stability}; every verdict
is independently recomputable from the ray data alone.

\subsection*{Acknowledgements}
The author thanks Milena Hering for bringing William Reynolds's thesis
\cite{Reynolds2023} to his attention and for pointing out its overlap
with the dimension-four results of the first version of this paper.

The author used LLMs during the development of this work for literature exploration, coding and editorial assistance and as an adversarial reader.

\section{Preliminaries}\label{sec:prelim}

\subsection{Smooth toric Fano varieties}
Let $N \cong \ZZ^n$ be a lattice, $M$ its dual. A smooth toric Fano
$n$-fold $X = X_\Sigma$ corresponds to a \emph{smooth Fano polytope}: the
convex hull, in $N_\QQ$, of the primitive ray generators
$u_\rho$ of $\Sigma$, which is reflexive with every facet a unimodular
simplex. The dual reflexive polytope
$P = \{ m \in M_\QQ : \langle m, u_\rho \rangle \ge -1 \ \forall \rho\}$
is the moment polytope of $-K_X$. For a facet
$F_\rho = P \cap \{ \langle \cdot, u_\rho\rangle = -1\}$ we write
$\Vol(F_\rho)$ for its normalized lattice volume, i.e.\ $(n-1)!$ times
its Euclidean volume with respect to the induced lattice; then
\begin{equation}\label{eq:degrees}
\deg_{-K}(D_\rho) \;=\; (-K_X)^{n-1}\cdot D_\rho \;=\; \Vol(F_\rho),
\qquad
\sum_\rho \Vol(F_\rho) \;=\; (-K_X)^n \;=\; n!\,\mathrm{vol}(P).
\end{equation}
By Mabuchi and Wang--Zhu \cite{Mabuchi1987,WangZhu2004}, $X$ admits a
K\"ahler--Einstein metric if and only if the barycenter of $P$ is the
origin. Throughout, $\dP_k$ denotes the del Pezzo surface obtained by
blowing up $\PP^2$ in $k$ points, torus-fixed for $k \le 3$; thus
$\dP_3$ is the smooth toric surface whose ray generators are the six
roots of $A_2$.

\subsection{Klyachko data of the tangent bundle}
An equivariant reflexive sheaf on $X_\Sigma$ is equivalent to a finite
collection of decreasing filtrations $\{E^\rho(j)\}_{j \in \ZZ}$ of a
fixed vector space $E$, one for each ray
\cite{Klyachko1990,Perling2004,Kool2011}. For the tangent bundle,
$E = N_\CC$ and the filtrations are two-step:
\[
E^\rho(j) \;=\;
\begin{cases}
N_\CC, & j \le 0,\\
\CC\, u_\rho, & j = 1,\\
0, & j \ge 2.
\end{cases}
\]
For a subspace $V \subseteq N_\CC$, the associated saturated equivariant
subsheaf $\mathcal{F}_V \subseteq T_X$ has filtrations
$V \cap E^\rho(j)$, whence by \eqref{eq:degrees}
\begin{equation}\label{eq:subsheafdeg}
\deg_{-K} \mathcal{F}_V \;=\; \sum_{\rho\,:\,u_\rho \in V}
\Vol(F_\rho),
\qquad
\mu(\mathcal{F}_V) \;=\; \frac{\deg_{-K}\mathcal{F}_V}{\dim V}.
\end{equation}

\begin{proposition}[{Klyachko; Kool; \cite[Prop.~1.2]{HNS2022}}]
\label{prop:criterion}
Let $X$ be a smooth toric Fano $n$-fold. Then $T_X$ is stable
(resp.\ semistable) with respect to $-K_X$ if and only if for every
subspace $0 \ne V \subsetneq N_\CC$ spanned by ray generators,
\[
\frac{1}{\dim V}\sum_{\rho\,:\,u_\rho \in V} \Vol(F_\rho)
\;<\; \frac{(-K_X)^n}{n}
\qquad (\text{resp.}\ \le).
\]
\end{proposition}

Two reductions underlie the statement, and we record them because the
polystability refinement reuses both. By the equivariant stability criterion
for toric sheaves \cite{Kool2011,HNS2022}, both stability and semistability may
be tested on saturated equivariant subsheaves. For instability this follows
immediately from uniqueness of the maximal Harder--Narasimhan subsheaf; the
equal-slope boundary needed for stability is the corresponding equivariant
socle statement. Thus it is enough to test subspaces
$V \subseteq N_\CC$. Second:

\begin{lemma}\label{lem:span}
For any subspace $V \subseteq N_\CC$ containing at least one ray
generator, let
$V' = \operatorname{span}\{u_\rho : u_\rho \in V\}$. Then
$\mu(\mathcal{F}_{V'}) \ge \mu(\mathcal{F}_V)$.
\end{lemma}

\begin{proof}
$V' \subseteq V$ contains exactly the same ray generators, so the
numerator of \eqref{eq:subsheafdeg} is unchanged while
$\dim V' \le \dim V$.
\end{proof}

Consequently it suffices to range over spans of linearly independent
subsets of ray generators, of which there are finitely many; and every
quantity in Proposition~\ref{prop:criterion} is an integer or a rational
number computed exactly from the polytope. Stability of $T_X$ on a
smooth toric Fano is therefore decidable by exact arithmetic, and the
decision comes with a proof: either the full list of subspace slopes
(for stability), or an explicit destabilizing subspace (for
instability).

\section{The computation and its validation}\label{sec:harness}

\subsection{Algorithm}
For each variety, presented by its ray generators: (i) enumerate the
vertices of the dual polytope $P$ exactly, verifying integrality
(reflexivity) along the way; (ii) compute each facet volume
$\Vol(F_\rho)$ by an exact recursive triangulation of the facet in a
lattice basis of its supporting hyperplane; (iii) enumerate the
ray-spanned subspaces $V$ (deduplicated by reduced row echelon form over
$\QQ$) with $1 \le \dim V \le n-1$, and compare each slope
\eqref{eq:subsheafdeg} with $(-K)^n/n$; (iv) record the maximal subsheaf
slope, all equality witnesses, and the exact barycenter of $P$. In an unstable
case the stored witness is the sum of all subspaces attaining the maximal
slope, hence the first Harder--Narasimhan subspace rather than an
enumeration-dependent maximizer. All arithmetic is over $\ZZ$ and $\QQ$
(fraction-free Bareiss elimination for determinants); no floating point
enters any verdict. Every variety is required to pass the internal
identity $\sum_\rho \Vol(F_\rho) = n!\,\mathrm{vol}(P)$ of
\eqref{eq:degrees}, with the right-hand side obtained by coning the computed
facet triangulations to the origin. This is a consistency identity, not an
independent certification of the triangulations. External checks are supplied
by the published-result and database comparisons below.

For the dimension-$7$ and $8$ members of the family of
Section~\ref{sec:roottwist}, which lie beyond the complete census above, the
reported degrees, slopes and barycenters are also evaluated directly from the
exact hexagon formulas of Lemma~\ref{lem:prism} and
Proposition~\ref{prop:reduction}: the relevant polynomials are integrated over
the fixed rational hexagon and its edges in rational arithmetic. An
accelerated full-dimensional implementation supplies an auxiliary cross-check.
It uses floating point only to propose vertex bases and Delaunay simplices,
then evaluates retained vertices and simplex determinants exactly. For the
reported inputs, an exhaustive exact basis audit confirmed that the proposal
stage omitted no true vertex. The accelerated and reference implementations
agree on every degree and verdict wherever both run (dimensions $\le 6$), and
both agree with the exact hexagon computation.

\subsection{Data}
The classification data (vertex sets of all smooth Fano polytopes in
dimensions $3$--$6$) were obtained from the \textsf{polyDB} database
(collection \textsf{Polytopes.Lattice.SmoothReflexive}, data of
Paffenholz computed by \O bro's algorithm \cite{Obro2007,polyDB}), stored
there in the dual convention and converted exactly (each facet normal
$[1, a]$ of the stored polytope contributes the vertex $a$ of the fan
polytope; no floating-point dualization). During ingestion every polytope was
re-verified as smooth and reflexive by exact arithmetic: all vertices are
primitive and every facet is a unimodular simplex at lattice distance one.
A validation script distributed with the code reconstructs every facet from the ray data and checks
these conditions, including that every listed ray is a genuine vertex; the
imported files also record the conversion protocol in their metadata. Counts match the
classification: $18$, $124$, $866$, $7622$. For dimensions $3$ and $5$
the vertex data were additionally compared file-by-file against
Paffenholz's original distribution and found identical. In dimensions $3$
and $4$, where the imported polyDB metadata include an independently stored
lattice volume, all $142$ values agree with the computed anticanonical
degrees.

\subsection{Validation against published results}\label{sec:validation}
The implementation reproduces the following results in its range:

\begin{enumerate}
\item \emph{Surfaces} (Fahlaoui \cite{Fahlaoui1989}; \cite[Thm.~1.3,
Cor.~3.1]{HNS2022}): $\PP^2$, $\dP_2$, $\dP_3$ stable; $\PP^1\times\PP^1$
and $\mathbb{F}_1$ strictly semistable, including both equality cases,
which test the semistable boundary of the criterion.
\item \emph{Fano three-folds} (Steffens \cite{Steffens1996}): of the
$18$ toric cases our computation finds $7$ unstable and $7$ strictly
semistable; these agree exactly with the toric members of Steffens's
lists, including
$\mathrm{Bl}_{\ell}\PP^3$ and the del~Pezzo products among the strictly
semistable cases.
\item \emph{Four-folds of Picard rank $\le 3$}
(\cite{BDGP2021}; \cite[Table~1]{DDK}): all $38$ varieties were matched
to our database entries by exact $GL(4,\ZZ)$-equivalence of fans
(necessary, since seven values of $(-K)^4$ are each shared by two to four
varieties), and all $38$ verdicts agree, as do the slopes
$\mu(T_X)$ reported in the proofs of \cite{DDK}.
\item \emph{All toric Fano four-folds} (Reynolds \cite[Appendix~A]{Reynolds2023}):
all $124$ stability classifications agree. We recomputed the facet volumes
and stability on the Macaulay2 database rays in Batyrev--Sato numbering
and matched them to our entries using the sorted facet-volume multisets,
which are distinct for all $124$ four-folds. In all $74$ unstable cases,
the rank and anticanonical degree of our first Harder--Narasimhan term
agree with those in Reynolds's full filtrations.
\item \emph{K\"ahler--Einstein tables}
(\cite{Mabuchi1987,Nakagawa1993,Nakagawa1994,NSY2023}): our barycenter
column
reproduces the known lists of K\"ahler--Einstein toric Fanos in
dimensions $\le 4$; in particular the rank-$\le 3$ four-fold list
$\{\PP^4, B_4, C_4, D_{13}\}$.
\item \emph{Hand-checked cases}: $\PP^n$, products, and projective
bundles in dimensions $2$--$4$, where all degrees and slopes are
classical.
\end{enumerate}

One near-miss shows why the matching in item (3) is necessary:
verdicts cannot be transported by the anticanonical degree alone,
and \cite{NSY2023} tabulate \emph{Ding} stability, a different notion;
three of the rank-$3$ four-folds ($G_4$--$G_6$) are Ding-unstable
with slope-stable tangent bundle. No contradiction is involved.

\section{The classification}\label{sec:results}

\begin{theorem}\label{thm:main}
The $(-K)$-slope stability of the tangent bundle of every smooth toric
Fano variety of dimension $3 \le n \le 6$ is as follows.
\begin{center}
\adjustbox{max width=\textwidth}{%
\begin{tabular}{lrrrr}
\toprule
$n$ & total & stable & strictly semistable & unstable \\
\midrule
3 & 18 & 4 & 7 & 7 \\
4 & 124 & 26 & 24 & 74 \\
5 & 866 & 109 & 82 & 675 \\
6 & 7622 & 636 & 366 & 6620 \\
\bottomrule
\end{tabular}}
\end{center}
\end{theorem}

\begin{proof}
By Proposition~\ref{prop:criterion} and Lemma~\ref{lem:span}, stability
is decided by the finitely many exact slope comparisons of
Section~\ref{sec:harness}, executed in integer and rational arithmetic
on the classified polytopes \cite{Obro2007,polyDB};
Section~\ref{sec:validation} records the validation.
\end{proof}

Complete tables (for each variety the database identifier, ray data,
facet degrees, the exact extremal slopes, a subspace attaining the
maximal slope, and the stability type) accompany the paper as ancillary
files and are also available, together with all code and input data, in the
public repository cited in the introduction.

The distribution by Picard rank exhibits the two phenomena announced in
the introduction: the dominance of instability, and the collapse of
stability at high rank. (In the tables, st, sss and un abbreviate
stable, strictly semistable and unstable.)

\begin{center}
\adjustbox{max width=\textwidth}{%
\begin{tabular}{lrrrr@{\qquad}rrrr}
\toprule
& \multicolumn{4}{c}{$n=4$} & \multicolumn{4}{c}{$n=5$} \\
$\rho$ & \# & st & sss & un & \# & st & sss & un \\
\midrule
1 & 1 & 1 & 0 & 0 & 1 & 1 & 0 & 0 \\
2 & 9 & 0 & 3 & 6 & 15 & 0 & 3 & 12 \\
3 & 28 & 6 & 3 & 19 & 91 & 9 & 5 & 77 \\
4 & 47 & 9 & 8 & 30 & 268 & 34 & 13 & 221 \\
5 & 27 & 7 & 4 & 16 & 312 & 46 & 23 & 243 \\
6 & 10 & 3 & 4 & 3 & 137 & 18 & 18 & 101 \\
7 & 1 & 0 & 1 & 0 & 35 & 1 & 15 & 19 \\
8 & 1 & 0 & 1 & 0 & 5 & 0 & 3 & 2 \\
9 & -- & -- & -- & -- & 2 & 0 & 2 & 0 \\
\bottomrule
\end{tabular}}
\end{center}

\noindent and for $n = 6$ (all $7622$ six-folds, with $\rho$ up to $12$)
the distribution is:

\begin{center}
\adjustbox{max width=\textwidth}{%
\begin{tabular}{lrrrr@{\qquad}lrrrr}
\toprule
$\rho$ & \# & st & sss & un & $\rho$ & \# & st & sss & un \\
\midrule
1 & 1 & 1 & 0 & 0 & 7 & 771 & 66 & 69 & 636 \\
2 & 26 & 0 & 4 & 22 & 8 & 186 & 10 & 46 & 130 \\
3 & 257 & 13 & 7 & 237 & 9 & 39 & 0 & 17 & 22 \\
4 & 1318 & 85 & 30 & 1203 & 10 & 11 & 0 & 8 & 3 \\
5 & 2807 & 246 & 75 & 2486 & 11 & 1 & 0 & 1 & 0 \\
6 & 2204 & 215 & 108 & 1881 & 12 & 1 & 0 & 1 & 0 \\
\bottomrule
\end{tabular}}
\end{center}

\noindent The complete $n=4$ table recovers Reynolds's classification
\cite[Appendix~A]{Reynolds2023}; its rows $\rho\le3$ also coincide with
\cite{BDGP2021,DDK}. The classifications for $n=5,6$ extend these
earlier results. Stability collapses at high rank in every
dimension: the last stable examples occur at $\rho = 6$ ($n=4$),
$\rho = 7$ ($n=5$) and $\rho = 8$ ($n=6$), and above these ranks the
varieties (pseudo-symmetric ones and products) are never stable (a
product tangent bundle splits). The unique stable five-fold with
$\rho = 7$ is the subject of Section~\ref{sec:roottwist}; its
higher-dimensional analogues populate the last stable rank in each
dimension.

\section{Polystability and Hermitian--Einstein metrics}
\label{sec:polystable}

Semistability does not decide the existence of a Hermitian--Einstein
metric; polystability does \cite{Donaldson1985,UhlenbeckYau1986}.
Polystability of an equivariant sheaf can, a priori, be exhibited by a
non-equivariant splitting; the first point is that on a toric variety it
never has to be (compare \cite{Klyachko1990,Kool2011}).

\begin{lemma}\label{lem:equivpoly}
Let $X$ be a smooth projective toric variety with torus $T$ and let $E$
be a $T$-equivariant reflexive sheaf that is polystable with respect to a
fixed polarization. Then $E$ admits a $T$-equivariant decomposition into
$T$-equivariant stable subsheaves of equal slope.
\end{lemma}

\begin{proof}
Write $E \cong \bigoplus_j S_j^{\oplus a_j}$ with the $S_j$ stable of
equal slope and pairwise non-isomorphic. For $t \in T$ we have
$t^* E \cong E$; torus translations preserve the polarization, so each
$t^* S_j$ is again a stable summand of $E$ of the same slope and
its isomorphism class lies in the finite set $\{[S_1], \dots, [S_r]\}$.
Each locus $\{t \in T : t^*S_j \cong S_{j'}\}$ is closed (it is where
$\dim \operatorname{Hom}(t^*S_j, S_{j'}) \ge 1$, and a nonzero map
between stable sheaves of equal slope is an isomorphism, so these loci
are disjoint); finitely many disjoint closed sets covering the
irreducible variety $T$ force one of them to be everything, and it
contains $t = 1$ for $j' = j$. Hence $t^* S_j \cong S_j$ for all $t$.
The isotypic summand is canonically
\[
B_j=\operatorname{im}\bigl(S_j\otimes
\operatorname{Hom}(S_j,E)\longrightarrow E\bigr).
\]
It therefore satisfies $t^* B_j=B_j$, is an equivariant subsheaf, and
$E = \bigoplus_j B_j$ equivariantly.

Next, each $S = S_j$ itself carries an equivariant structure. The pairs
$(t, \psi)$ with $\psi \colon t^* S \xrightarrow{\ \sim\ } S$ form
an affine algebraic group $\widetilde{T}$, a central extension
$1 \to \mathbb{G}_m \to \widetilde{T} \to T \to 1$ because $S$ is simple;
$S$ is tautologically $\widetilde{T}$-equivariant. Every such extension
of a torus splits: the commutator pairing $T \times T \to \mathbb{G}_m$
is bimultiplicative, so $t \mapsto e(t, \cdot)$ is a morphism from the
connected group $T$ to the discrete character lattice, hence trivial;
$\widetilde{T}$ is therefore a commutative extension of tori, hence a torus,
and the extension splits because its
character sequence
$0 \to X^*(T) \to X^*(\widetilde{T}\,) \to X^*(\mathbb{G}_m) \to 0$
has free quotient $X^*(\mathbb{G}_m) \cong \ZZ$. A splitting $T \to \widetilde{T}$ makes $S$ itself
$T$-equivariant.

Finally $W_j = \operatorname{Hom}(S_j, B_j)$ is a finite-dimensional
$T$-representation, and evaluation
$S_j \otimes W_j \to B_j$ is an equivariant isomorphism; decomposing
$W_j$ into characters $\chi$ decomposes $B_j$ equivariantly into the
sheaves $S_j \otimes \chi$, each equivariant with underlying sheaf
$S_j$, hence stable of the same slope.
\end{proof}

For toric tangent bundles the equivariant question is then purely
combinatorial, because of the following elementary observation.

\begin{proposition}\label{prop:split}
Let $X$ be a smooth toric variety. Equivariant direct-sum decompositions
$T_X \cong \bigoplus_{i=1}^m \mathcal{F}_{V_i}$ correspond exactly to
decompositions $N_\CC = \bigoplus_{i=1}^m V_i$ such that every ray
generator $u_\rho$ lies in some $V_i$. If $X$ is in addition projective and
Fano, then $T_X$ is polystable
with respect to $-K_X$ if and only if the ray generators admit a
partition into groups $\Gamma_1, \dots, \Gamma_m$ whose spans $V_i =
\operatorname{span} \Gamma_i$ are linearly independent with
$\sum_i \dim V_i = n$, such that each $\mathcal{F}_{V_i}$ has slope
$(-K)^n/n$ and each is stable, i.e.\ every ray-spanned proper subspace
$0 \neq U \subsetneq V_i$ satisfies
$\mu(\mathcal{F}_U) < (-K)^n/n$.
\end{proposition}

\begin{proof}
A direct sum of equivariant reflexive sheaves has, on each ray, the
direct sum of the filtrations. Since the filtrations of $T_X$ are
two-step, a decomposition $N_\CC = \bigoplus V_i$ is compatible with them
if and only if for every $\rho$ the line $\CC u_\rho$ equals
$\bigoplus_i (V_i \cap \CC u_\rho)$, i.e.\ $u_\rho \in V_i$ for some
(necessarily unique) $i$; in that case the summand with fiber $V_i$ has
exactly the filtrations of $\mathcal{F}_{V_i}$, and Klyachko data
determine the sheaf. The polystability statement follows since a
polystable bundle is a direct sum of stable sheaves of equal slope, the
decomposition may be taken equivariant by Lemma~\ref{lem:equivpoly}, and
stability of each summand is tested by the same equivariant stability
criterion as above and then, by Lemma~\ref{lem:span}, on ray-spanned
subspaces of $V_i$.
\end{proof}

The splitting has a canonical matroid-theoretic and geometric form.

\begin{corollary}[Ray-matroid and product criterion]\label{cor:matroidsplit}
Let $X=X_\Sigma$ be a smooth toric Fano variety, and let
$\mathsf M_\Sigma$ be the linear matroid represented by the ray generators.
Write $C_1,\dots,C_s$ for its connected components, where two rays are
equivalent if they can be joined by a chain of minimal linear dependences,
and put
\[
V_a=\operatorname{span}_\CC\{u_\rho:\rho\in C_a\}.
\]
Then $N_\CC=\bigoplus_aV_a$, and this is the finest decomposition compatible
with all ray lines. It is induced by a product decomposition
\[
X\cong\prod_{a=1}^sX_a,
\qquad
T_X\cong\bigoplus_{a=1}^s p_a^*T_{X_a},
\]
where each $X_a$ is a smooth toric Fano variety with fan in
$N\cap V_a$. Every other equivariant direct-sum decomposition of $T_X$ is
obtained by grouping some of these factors.

If $T_X$ is semistable, every component summand has slope $\mu(T_X)$ and is
semistable. Moreover, $T_X$ is polystable if and only if every $T_{X_a}$ is
$(-K_{X_a})$-stable. Consequently, a strictly semistable $T_X$ is polystable
if and only if $X$ is a nontrivial product of smooth toric Fano varieties
with stable tangent bundles.
\end{corollary}

\begin{proof}
A circuit cannot meet two summands of a direct-sum decomposition containing
every ray in one summand: projecting its dependence to either summand would
give a dependence on a proper subset of the circuit. Thus every $C_a$ lies in
one summand of every decomposition in Proposition~\ref{prop:split}.
Conversely, a linear dependence meeting two components, chosen with minimal
support, would be such a circuit. Hence the $V_a$ are linearly independent;
since the rays span $N_\CC$, their direct sum is all of $N_\CC$. This proves
the canonical splitting and its fineness.

Choose a maximal cone of $\Sigma$. Its primitive generators form a lattice
basis, and their subsets in the spaces $V_a$ show that
$N=\bigoplus_a(N\cap V_a)$. Every cone of $\Sigma$ is the direct sum of its
intersections with the $V_a$. Because intersections and faces of cones are
componentwise under this direct-sum decomposition, the collections
$\Sigma_a=\{\sigma\cap V_a:\sigma\in\Sigma\}$ are fans. They are complete:
apply completeness of $\Sigma$ to a point of $V_a$. Moreover every combination
of maximal cones of the $\Sigma_a$ occurs in $\Sigma$. Indeed, choose an
interior point in each and apply completeness to their sum; uniqueness of the
component decomposition identifies the resulting intersections. Thus
$\Sigma=\prod_a\Sigma_a$ and $X=\prod_aX_a$. Each factor is smooth and Fano,
and its tangent pullback is the summand $\mathcal F_{V_a}$.

Set $n_a=\dim X_a$ and $d_a=(-K_{X_a})^{n_a}$. For an equivariant subsheaf
$\mathcal G\subseteq T_{X_a}$, the product intersection formula gives
\[
\deg_{-K_X}(p_a^*\mathcal G)=
\frac{(n-1)!}{(n_a-1)!\prod_{b\ne a}n_b!}
\left(\prod_{b\ne a}d_b\right)\deg_{-K_{X_a}}(\mathcal G).
\]
By the Klyachko filtrations, every saturated equivariant subsheaf of
$p_a^*T_{X_a}$ used in the stability test is the pullback of the corresponding
subsheaf of $T_{X_a}$. Thus the strict slope comparisons inside the $a$-th
summand are exactly those for $T_{X_a}$, while
\[
\mu_{-K_X}(p_a^*T_{X_a})=
\frac{(n-1)!}{\prod_b n_b!}\prod_b d_b
\]
is independent of $a$. Consequently, stable factor tangents give a polystable
$T_X$.

Suppose now that $T_X$ is semistable. Each component summand has slope at
most $\mu(T_X)$; since their ranks and degrees add to those of $T_X$, every
component slope equals $\mu(T_X)$. Any subsheaf destabilizing a component would
then destabilize $T_X$, so every component is semistable.

Conversely, suppose $T_X$ is polystable. Lemma~\ref{lem:equivpoly} supplies an equivariant
decomposition into stable equal-slope summands. By fineness, each is a union
of matroid components, but it cannot contain two: either component would be a
proper subsheaf of the same slope. Hence the stable summands are precisely the
component summands, and the factor tangent bundles are stable by the scaling
formula above.
\end{proof}

The census was computed by the finite search over ray partitions of
Proposition~\ref{prop:split}; Corollary~\ref{cor:matroidsplit} gives an
equivalent canonical test using the connected components of the ray matroid.
The two tests agree on every strictly semistable case. Applying them to
Theorem~\ref{thm:main} gives:

\begin{theorem}\label{thm:HE}
Among the strictly semistable varieties, exactly $4$ of $7$ ($n=3$),
$14$ of $24$ ($n=4$), $52$ of $82$ ($n=5$), and $259$ of $366$ ($n=6$) are
polystable. Combined with the stable count, the tangent bundle admits a
Hermitian--Einstein metric with respect to a K\"ahler form representing
$c_1(X)$ for exactly
\[
8 \text{ of } 18 \ (n=3), \qquad
40 \text{ of } 124 \ (n=4), \qquad
161 \text{ of } 866 \ (n=5), \qquad
895 \text{ of } 7622 \ (n=6)
\]
smooth toric Fano varieties.
\end{theorem}

\begin{proof}
By Lemma~\ref{lem:equivpoly}, Proposition~\ref{prop:split}, and
Corollary~\ref{cor:matroidsplit},
polystability of each strictly semistable $T_X$ is decided by the
equivalent splitting tests above, executed in exact arithmetic; combining the
resulting counts with the stable counts of Theorem~\ref{thm:main} and
the Kobayashi--Hitchin correspondence
\cite{Donaldson1985,UhlenbeckYau1986} gives the totals.
\end{proof}

\begin{example}[Semistable but not polystable]
The semistable-but-not-polystable class is the exact locus where
semistability fails to deliver a metric, and its members are structured:
\begin{enumerate}
\item The blow-up $\mathrm{Bl}_\ell \PP^3$ of $\PP^3$ along a line is strictly
semistable but not polystable: the three-dimensional analogue of
the Hirzebruch surface $\mathbb{F}_1$, whose equality subsheaf is the
relative tangent sheaf of the $\PP^1$-fibration structure.
\item $\mathbb{F}_1 \times \PP^1$ is a \emph{product} that is not
polystable:
its equality decomposition requires the factor bundle $T_{\mathbb{F}_1}$,
which
is itself only strictly semistable. Products of semistables are
semistable, but polystability demands stable summands.
\item By contrast $\dP_2 \times \PP^1$ \emph{is} polystable
($T_{\dP_2}$ is stable by \cite{Fahlaoui1989}, and the slopes balance),
so it carries a Hermitian--Einstein tangent bundle although it has no
K\"ahler--Einstein metric.
\end{enumerate}
\end{example}

\section{Comparison with the K\"ahler--Einstein condition}
\label{sec:KE}

For toric Fanos the K\"ahler--Einstein condition is the vanishing of the
barycenter of $P$ \cite{Mabuchi1987,WangZhu2004}, which we compute
exactly for every variety. The comparison runs in two directions.

\emph{K\"ahler--Einstein $\Rightarrow$ polystable.} This is the
Kobayashi--L\"ubke direction \cite{Kobayashi1987,Lubke1983}, and the data
confirm it without exception: no K\"ahler--Einstein variety in the census
is unstable, and every strictly semistable K\"ahler--Einstein variety is
polystable ($3/3$ in dimension $3$, $8/8$ in dimension $4$, $16/16$ in
dimension $5$, $38/38$ in dimension $6$). We emphasize that this is a
useful end-to-end comparison with the independently published
K\"ahler--Einstein classifications in low dimensions. Internally, the
barycenter and slope tests are different functions of the same exact
polytope data. It is also the effective form of the
subspace concentration conditions of \cite{HNS2022,HenkLinke2014}; for
their affine refinements see \cite{Wu2022,FreyerHenkKipp}.

\emph{Polystable $\not\Rightarrow$ K\"ahler--Einstein, overwhelmingly in the
census.}
Fahlaoui's example $\dP_2$ is the smallest instance of what our census
shows to be the dominant behaviour:
\begin{center}
\adjustbox{max width=\textwidth}{%
\begin{tabular}{lrrr}
\toprule
$n$ & stable & of which KE & stable non-KE \\
\midrule
2 & 3 & 2 & 1 \\
3 & 4 & 2 & 2 \\
4 & 26 & 4 & 22 \\
5 & 109 & 7 & 102 \\
6 & 636 & 13 & 623 \\
\bottomrule
\end{tabular}}
\end{center}
In dimension five, $94\%$ of the varieties with stable tangent bundle
admit no K\"ahler--Einstein metric, and in dimension six the fraction
rises to $98\%$ ($623$ of $636$): the K\"ahler--Einstein proportion of
the stable locus continues to decay with dimension. A Hermitian--Einstein
metric on $T_X$ with respect to an anticanonical K\"ahler form is thus a far
weaker, and
far more common, Einstein condition than a K\"ahler--Einstein metric on
$X$, already within the toric world. Counting stable and polystable tangent
bundles together, the number of varieties with this Hermitian--Einstein
metric but without a K\"ahler--Einstein metric is $3$, $28$, $138$, and
$844$ in dimensions $3$, $4$, $5$, and $6$, respectively.

\section{The root-twisted \texorpdfstring{$\dP_3$}{dP3}-fibrations}
\label{sec:roottwist}

Among the $42$ smooth toric Fano five-folds of Picard rank $\ge 7$,
exactly one, polyDB entry \textsf{F.5D.0611}, has stable tangent
bundle. It has $12$ rays, $48$
maximal cones, $(-K)^5 = 2580$, facet degrees
$(278^{\times 6}, 156^{\times 3}, 148^{\times 3})$, and
$\mu(T_X) = 516$ against a maximal subsheaf slope of $506$; its
barycenter is zero, so it is K\"ahler--Einstein. Its structure is revealed by its
primitive collections, all of cardinality two: six of its rays lie in a
rank-two sublattice, forming the $A_2$ root hexagon
$\{\pm\alpha, \pm\beta, \pm(\alpha+\beta)\}$ (the fan of $\dP_3$),
and the projection along that plane maps the remaining six rays onto the
fan of $(\PP^1)^3$. The fibration is not a product: the three base
directions lift with twists $\alpha$, $\beta$, $-(\alpha+\beta)$ in the
fiber plane.

\begin{definition}\label{def:roottwist}
Write $R = \{\pm\alpha, \pm\beta, \pm(\alpha{+}\beta)\}$ for the six
roots, where $\alpha, \beta$ is a basis of a rank-two lattice $N_f$.
Fix $k \ge 0$ and twists $t_1, \dots, t_k \in \{0\} \cup R$, and let
$f_i$ be the $i$-th standard basis vector of $\ZZ^k$. Put
$u_i^+=f_i+t_i$ and $u_i^-=-f_i$. In $N_f\oplus\ZZ^k$, let
$\Sigma(t_\bullet)$ be the fan whose maximal cones are
\[
\operatorname{cone}(\rho,\rho',u_1^{\varepsilon_1},\dots,
u_k^{\varepsilon_k}),
\qquad \varepsilon_i\in\{+,-\},
\]
where $\rho,\rho'$ run over adjacent rays of the hexagonal fan of $\dP_3$.
Define $X(t_1,\dots,t_k)=X_{\Sigma(t_\bullet)}$, of dimension $n=2+k$.
Lemma~\ref{lem:prism} below verifies directly that these cones form the
smooth complete normal fan of a reflexive polytope.
\end{definition}

Each $X(t_1,\dots,t_k)$ is a smooth toric variety fibered over
$(\PP^1)^k$ with fiber $\dP_3$; the untwisted case is
$X(0,\dots,0) = \dP_3 \times (\PP^1)^k$. Every $X(t_\bullet)$ is in fact a
smooth toric \emph{Fano} variety of Picard rank $n + 2$
(Lemma~\ref{lem:prism} below). The variety above is
$X(\alpha, \beta, -(\alpha+\beta))$: its twist multiset is a full orbit of
the order-three rotation $\tau\colon \alpha \mapsto \beta \mapsto
-(\alpha+\beta)$ of the root hexagon, a symmetry that extends to an
automorphism of the variety and forces the barycenter of the moment
polytope to the origin (see the proof of Theorem~\ref{thm:families}), so
$X$ is K\"ahler--Einstein.

The computations to date are all exact: dimensions $\le 6$ use the reference
implementation of Section~\ref{sec:harness}, and every row is independently
reproduced by the rational hexagon formulas below.

\begin{center}
\adjustbox{max width=\textwidth}{%
\begin{tabular}{llrlll}
\toprule
twists & $n$ & $(-K)^n$ & stability & KE & remark \\
\midrule
$(\alpha,-\alpha)$ & 4 & 268 & stable & yes &
the smallest zero-sum pair \\
$(\alpha,\beta)$ & 4 & 278 & stable & \textbf{no} &
the smallest unbalanced stable example \\
$(\alpha,\beta,-(\alpha{+}\beta))$ & 5 & 2580 & stable & yes &
$= $ \textsf{F.5D.0611}, the $\rho{\ge}7$ example \\
$(0,0,0)$ & 5 & 2880 & strictly ss & yes & the product \\
$(\alpha,-\alpha,\beta,-\beta)$ & 6 & 29928 & stable & yes &
zero-sum pairs \\
$(\alpha,\beta,-(\alpha{+}\beta),0)$ & 6 & 30960 & strictly ss & yes &
one untwisted factor \\
$(\alpha,\beta,-(\alpha{+}\beta),\alpha)$ & 6 & 30960 & strictly ss &
no & twist sum $\alpha \ne 0$ \\
$(\alpha,\beta,-(\alpha{+}\beta),\alpha,-\alpha)$ & 7 & 404544 & stable &
\textbf{no} & zero-sum, no fiber symmetry \\
$(\alpha,\beta,-(\alpha{+}\beta))^{\times 2}$ & 8 & 6242544 & stable &
yes & two $\tau$-orbits \\
\bottomrule
\end{tabular}}
\end{center}

\begin{theorem}[zero-sum stability]\label{thm:zerosum}
Let $t_1,\dots,t_k\in R$ be nonzero roots with $\sum_i t_i=0$. Then
$T_{X(t_\bullet)}$ is $(-K)$-stable. In particular, smooth toric Fano
varieties of Picard rank $n+2$ with stable
tangent bundle exist in every dimension $n\ge4$.
\end{theorem}

\begin{theorem}[K\"ahler--Einstein classification]\label{thm:roottwist-ke}
Among the zero-sum cases of Theorem~\ref{thm:zerosum},
$X(t_\bullet)$ is K\"ahler--Einstein if and only if the
twist multiset $\{t_1, \dots, t_k\}$ is invariant under negation
$t \mapsto -t$ or under the order-three rotation
$\tau\colon \alpha \mapsto \beta \mapsto -(\alpha+\beta)$ of the root hexagon.
\end{theorem}

Two mechanisms underlie the statements. \emph{Stability:} an untwisted factor
leaves an equal-slope summand, while for a balanced collection of nonzero root
twists every proper ray-spanned subspace generates a subsheaf of slope
strictly below $\mu(T_X)$ (Theorem~\ref{thm:zerosum}). The twist is thus a geometric counterpart of the stable
perturbations of semistable tangent sheaves that Napame and Tipler
obtain by varying the polarization or blowing up
\cite{NapameTipler2024}: here the polarization remains anticanonical,
and it is the fibration itself that is deformed.
\emph{The K\"ahler--Einstein criterion:} whenever the twist multiset is
preserved by negation or by the rotation $\tau$, that symmetry extends to
an automorphism of $X(t_\bullet)$, permuting the base factors. The
barycenter of the moment polytope is fixed by the induced map, whose
action on the fibre component is negation resp.\ the rotation (neither
fixes a nonzero vector), so the fibre component of the barycenter
vanishes; and the base components vanish with it, being proportional to it
(Lemma~\ref{lem:prism}(3) below). This proves sufficiency; necessity follows
from the extreme-exponent ordering in Lemma~\ref{lem:moment-order} below.
The statement also
corrects the naive expectation that a
vanishing twist \emph{sum} alone forces the barycenter to zero. The minimal
counterexample to that expectation is the dimension-seven variety
$X(\alpha,\beta,-(\alpha+\beta),\alpha,-\alpha)$: its twists are nonzero
roots summing to zero, but the multiset is invariant under neither symmetry,
and the variety is $(-K)$-stable with a \emph{nonzero} barycenter: a
stable, non-K\"ahler--Einstein tangent bundle beyond the complete census
carried out here.

For comparison, exact computation through dimension eight gives nine
zero-sum classes of nonzero roots up to the hexagon's dihedral symmetry and
reproduces Theorems~\ref{thm:zerosum} and~\ref{thm:roottwist-ke} in every
case ($\tau$-invariance and negation-invariance are marked $\checkmark$):

\begin{center}
\adjustbox{max width=\textwidth}{%
\begin{tabular}{llccc}
\toprule
$n$ & twist multiset & neg-inv. & $\tau$-inv. & K\"ahler--Einstein \\
\midrule
4 & $\alpha,-\alpha$ & $\checkmark$ & & yes \\
5 & $\alpha,\beta,-(\alpha{+}\beta)$ & & $\checkmark$ & yes \\
6 & $\alpha,-\alpha,\alpha,-\alpha$ & $\checkmark$ & & yes \\
6 & $\alpha,\beta,-\alpha,-\beta$ & $\checkmark$ & & yes \\
7 & $\alpha,\alpha,\beta,-\alpha,-(\alpha{+}\beta)$ & & & \textbf{no} \\
8 & $\alpha,\alpha,\alpha,-\alpha,-\alpha,-\alpha$ & $\checkmark$ & & yes \\
8 & $\alpha,\alpha,\beta,\beta,-(\alpha{+}\beta),-(\alpha{+}\beta)$ & & $\checkmark$ & yes \\
8 & $\alpha,\alpha,\beta,-\alpha,-\alpha,-\beta$ & $\checkmark$ & & yes \\
8 & $\alpha,\beta,\alpha{+}\beta,-\alpha,-\beta,-(\alpha{+}\beta)$ & $\checkmark$ & $\checkmark$ & yes \\
\bottomrule
\end{tabular}}
\end{center}

\noindent K\"ahler--Einstein holds in exactly the eight classes with a
symmetry and fails in the single class (dimension seven) with neither.

\subsection*{The prism structure and the stability criterion}
The moment polytopes of the family are explicit enough to begin the proof
of Theorem~\ref{thm:zerosum}. Write $Q \subset M_f \otimes \QQ$,
with $M_f$ the lattice dual to $N_f$, for
the hexagon $\{x : \langle x, \rho\rangle \ge -1 \text{ for all six roots }
\rho\}$ (the anticanonical polytope of $\dP_3$; lattice area $3$, six edges
$E_\rho \subset \{\langle x,\rho\rangle = -1\}$ of lattice length one), set
$\ell_t(x) = \langle x, t\rangle$, and
\[
F(x) \;=\; \prod_{i=1}^{k}\bigl(2 + \ell_{t_i}(x)\bigr).
\]

\begin{lemma}[prism structure]\label{lem:prism}
Let $t_1, \dots, t_k \in \{0\} \cup R$ and
$n = 2 + k$.
\begin{enumerate}
\item $X(t_1, \dots, t_k)$ is a smooth toric Fano variety of Picard rank
$n + 2$, with anticanonical polytope the generalized prism
\[
P \;=\; \bigl\{(x, y) \in (M_f \oplus \ZZ^k) \otimes \QQ \;:\;
x \in Q, \;\; -1 - \ell_{t_i}(x) \,\le\, y_i \,\le\, 1 \bigr\},
\]
whose fibre over $x \in Q$ is a box with side lengths
$2 + \ell_{t_i}(x) \in [1, 3]$.
\item The anticanonical degrees are hexagon integrals:
\[
(-K)^n = n! \int_Q F, \quad
\deg D_{f_i + t_i} = \deg D_{-f_i} = (n{-}1)! \int_Q \frac{F}{2 + \ell_{t_i}},
\quad
\deg D_\rho = (n{-}1)! \int_{E_\rho} F,
\]
with the lattice measure on $E_\rho$; in particular
$\mu(T_X) = (n-1)! \int_Q F$.
\item The barycenter $b = (b_x, b_y)$ of $P$ has fibre component
$b_x = \int_Q x F \,/ \int_Q F$ and base components
$b_{y_i} = -\tfrac12 \langle b_x, t_i \rangle$. In particular
$X(t_\bullet)$ is K\"ahler--Einstein if and only if the hexagon moment
$\int_Q x\, F(x)\, dx$ vanishes.
\end{enumerate}
\end{lemma}

\begin{proof}
(1) The inequalities $\langle (x,y), u\rangle \ge -1$, over the rays of
Definition~\ref{def:roottwist}, read $x \in Q$ for the six roots and
$-1 - \ell_{t_i}(x) \le y_i \le 1$ for the pair $f_i + t_i$, $-f_i$; this
is $P$. Its vertices are integral: over a vertex $v$ of $Q$ the box
endpoints are $1$ and $-1 - \ell_{t_i}(v)$ with
$\ell_{t_i}(v) \in \{0, \pm 1\}$, and the two never coincide since the
widths lie in $[1,3]$. Every defining inequality supports a facet (a
prism over an edge $E_\rho$ for the root inequalities; the graph of an
affine function over the prism with the $i$-th factor omitted, for the
base inequalities), so $P$ is reflexive with the rays of
Definition~\ref{def:roottwist} as its facet normals. The normal cone at
each of the $6 \cdot 2^k$ vertices is spanned by two adjacent hexagon
roots together with one of $f_i + t_i$, $-f_i$ per factor, a unimodular
basis; hence the normal fan of $P$ is the smooth complete fan of
Definition~\ref{def:roottwist} and $X(t_\bullet)$ is Fano, of Picard rank
$(6 + 2k) - n = n + 2$.
(2) Integrating the box volume gives $\mathrm{vol}(P)$, hence $(-K)^n$ by
\eqref{eq:degrees}. The facet $\{y_i = 1\}$ projects
isomorphically (dropping $y_i$) to the prism with the $i$-th factor
omitted, and the facet $\{y_i = -1 - \ell_{t_i}\}$ is carried to it by the
unimodular shear $y_i \mapsto -y_i - \ell_{t_i}(x)$, giving the base
degrees; the facet over $E_\rho$ is the same prism construction over the
edge. All identifications are lattice isomorphisms, so normalized volumes
agree.
(3) The $y_i$-fibre over $x$ is the interval
$[-1 - \ell_{t_i}(x),\, 1]$, with midpoint $-\ell_{t_i}(x)/2$; hence
$b_{y_i} = -\tfrac12 \int_Q \ell_{t_i} F / \int_Q F =
-\tfrac12 \langle b_x, t_i\rangle$ by linearity of $\ell_{t_i}$. So
$b = 0$ if and only if $b_x = 0$, and the K\"ahler--Einstein claim is
Wang--Zhu \cite{Mabuchi1987,WangZhu2004}.
\end{proof}

The comparisons of Proposition~\ref{prop:criterion} now become inequalities
between hexagon integrals. The bridge is the divergence theorem applied to
the field $x F$ on the reflexive polygon $Q$: since every facet of $Q$ lies
on $\{\langle x, \rho \rangle = -1\}$ with $\rho$ primitive, the Euclidean
boundary factors cancel and
\begin{equation}\label{eq:divergence}
\int_{\partial Q} F \;=\; 2 \int_Q F \;+\; \int_Q \langle x, \nabla F
\rangle, \qquad
\langle x, \nabla F\rangle = F \sum_{i=1}^k
\frac{\ell_{t_i}}{2 + \ell_{t_i}},
\end{equation}
with the boundary integral taken in the lattice measure. Writing
$\Phi(t_\bullet) = \int_Q F \sum_i \ell_{t_i}/(2+\ell_{t_i})$, the fibre
subsheaf $\mathcal{F}_{N_f}$ (the relative tangent sheaf of
$\pi \colon X \to (\PP^1)^k$, and the subsheaf whose slope equality makes
the untwisted product merely polystable) satisfies, by
Lemma~\ref{lem:prism} and \eqref{eq:divergence},
\begin{equation}\label{eq:phigap}
\mu(\mathcal{F}_{N_f}) - \mu(T_X) \;=\; \tfrac{(n-1)!}{2}\, \Phi(t_\bullet).
\end{equation}
Call a proper ray-spanned subspace $V$, or the subsheaf $\mathcal{F}_V$,
\emph{critical} if $\mu(\mathcal{F}_V) = \mu(T_X)$ and \emph{subcritical}
if $\mu(\mathcal{F}_V) < \mu(T_X)$; by
Proposition~\ref{prop:criterion}, stability of $T_X$ is the statement that
every proper ray-spanned subspace is subcritical.

\begin{proposition}\label{prop:partialstab}
Assume $k\ge1$ and let every $t_i$ lie in $\{0\} \cup R$.
\begin{enumerate}
\item If some $t_i = 0$, then $\mu(\mathcal{F}_{\CC f_i}) = \mu(T_X)$
exactly; hence $T_X$ is not stable.
\item If the $t_i$ are all equal to a single nonzero root, then
$\Phi \ge 0$, strictly for $k \ge 2$; hence $T_X$ is not stable, and is
unstable for $k \ge 2$.
\item If the $t_i$ are nonzero roots with $\sum_i t_i = 0$, then
$\Phi < 0$: the fibre subsheaf is subcritical.
\end{enumerate}
\end{proposition}

\begin{proof}
(1) With $t_i = 0$ the line $V = \CC f_i$ contains exactly the two rays
$\pm f_i$, each of degree $(n-1)! \int_Q F/2$ by Lemma~\ref{lem:prism}, so
$\mu(\mathcal{F}_V) = (n-1)! \int_Q F = \mu(T_X)$.
(2) With all $t_i = t$, $\Phi = k \int_Q (2+\ell_t)^{k-1} \ell_t$; folding
the centrally symmetric $Q$ along $x \mapsto -x$ gives
$\Phi = k \int_{Q \cap \{\ell_t > 0\}} \ell_t \bigl[(2+\ell_t)^{k-1} -
(2-\ell_t)^{k-1}\bigr] \ge 0$, strict for $k \ge 2$; conclude by
\eqref{eq:phigap}.
(3) The function $u \mapsto u/(2+u)$ is strictly concave for $u > -2$, so
pointwise on $Q$, by Jensen,
\[
\sum_i \frac{\ell_{t_i}(x)}{2 + \ell_{t_i}(x)} \;\le\;
k \, \frac{\bar\ell(x)}{2 + \bar\ell(x)}, \qquad
\bar\ell = \tfrac{1}{k} \, \ell_{\sum_i t_i} = 0,
\]
with equality only where all $\ell_{t_i}(x)$ coincide. A zero-sum multiset
of nonzero roots contains two distinct roots, so the equality locus is
contained in a proper linear subspace and has measure zero. Hence
$\Phi < 0$ strictly since $F \ge 1$ on $Q$.
\end{proof}

Proposition~\ref{prop:partialstab}(1) proves that an untwisted factor always
obstructs stability, while (2) gives a simple class of unbalanced unstable
examples. Part (3) settles, in the zero-sum case, the
slope inequality of the subsheaf that dominates the statistics of
Section~\ref{sec:structure}. In fact the entire stability question reduces
to finitely many hexagon integrals, in a form that proves
Theorem~\ref{thm:zerosum} for every zero-sum multiset. For the $i$-th twist
define its \emph{singleton integral}
$\varphi_i = \int_Q F\, \ell_{t_i}/(2 + \ell_{t_i})$ (so
$\Phi = \sum_i \varphi_i$), and for each of the three root lines
$\{\pm\rho\}$ define the \emph{line share}
\[
S_\rho \;=\; \int_{E_\rho} F \;+\; \int_{E_{-\rho}} F
\;+\; \sum_{i \,:\, t_i = \pm\rho} \max(0,\, -\varphi_i).
\]

\begin{proposition}[reduction]\label{prop:reduction}
Let all $t_i$ be nonzero roots. Then $T_X$
is $(-K)$-stable if and only if
\begin{enumerate}
\item[(a)] $\varphi_i < 0$ for every $i$, and
\item[(b)] $S_\rho < \int_Q F$ for each of the three root lines.
\end{enumerate}
Moreover, if (a) holds then $S_\alpha + S_\beta + S_{\alpha+\beta} =
2\int_Q F$, so (b) states exactly that the three line shares satisfy the
strict triangle inequality.
\end{proposition}

\begin{proof}
Every ray-spanned subspace $V$ meets $N_{f,\CC}$ in $0$, a root line, or
$N_{f,\CC}$: a single lift $f_i + t_i$ or $-f_i$ contributes no fibre
direction to a span, while the pair contributes exactly $\CC t_i$.

\emph{Single lifts never help.} If $V$ contains single lifts, its ray
set splits off theirs, and $\mu(\mathcal{F}_V)$ is a weighted mean of
the slope of the remaining part and of base degrees
$(n-1)!\int_Q F/(2+\ell_{t_j})$, each strictly below $\mu(T_X)$ because
$2 + \ell_{t_j} > 1$ off a set of measure zero. In particular a subspace
spanned by single lifts alone (the case $V \cap N_{f,\CC} = 0$) is
subcritical. Thus, for testing stability, it suffices to consider
configurations built from pairs and roots only.

\emph{Case $V \supseteq N_{f,\CC}$,} with pair set
$A \subsetneq \{1, \dots, k\}$: Lemma~\ref{lem:prism} and
\eqref{eq:divergence} turn $\mu(\mathcal{F}_V) < \mu(T_X)$ into
$\sum_{i \notin A} \varphi_i < 0$. As $A$ varies these conditions are
equivalent, termwise, to (a): taking $A$ to omit a single index $j$
shows $\varphi_j < 0$ is necessary, and conversely under (a) every sum
$\sum_{i \notin A} \varphi_i$ of negative terms is negative.

\emph{Case $V \cap N_{f,\CC} = \CC\rho$,} with pair set
$A \subseteq \{i : t_i = \pm\rho\}$: the condition reads
$\int_{E_\rho}F + \int_{E_{-\rho}}F - \sum_{A}\varphi_i < \int_Q F$, and
the worst $A$ gives (b).

\emph{The share identity.} The six edge integrals sum to
$\int_{\partial Q} F$, every twist lies on exactly one root line, and
under (a) the corrections total $-\Phi$; so \eqref{eq:divergence} gives
$\sum_\rho S_\rho = \bigl(2\int_Q F + \Phi\bigr) - \Phi = 2\int_Q F$.
\end{proof}

In $X(\alpha, \beta, -(\alpha{+}\beta))$, for instance, the maximal slope
$506$ against $\mu = 516$ is the $N_f$-configuration omitting one twist,
and the gap $10$ equals $-6\varphi_1$ (the three $\varphi_i$ coincide by
rotation symmetry). We now prove that the criterion is decisive for every
zero-sum multiset.

Put
\[
 \gamma_1=\alpha,\qquad \gamma_2=\beta,\qquad
 \gamma_3=-(\alpha+\beta),\qquad x_j=\ell_{\gamma_j};
\]
thus $x_1+x_2+x_3=0$ and $Q=\{|x_j|\le1\}$. If $p_j,q_j$ are the
multiplicities of $\gamma_j,-\gamma_j$, respectively, then the twist sum
vanishes exactly when
\[
 p_1-q_1=p_2-q_2=p_3-q_3=:d.
\]
Global negation of the roots is induced by the automorphism $-1$ of the
hexagon and gives an isomorphic root-twist variety. Applying it if necessary,
we may assume $d\ge0$. Hence
every nonempty zero-sum multiset has the unique normal form
\begin{equation}\label{eq:zerosum-normal}
 F=H^d\prod_{j=1}^3(4-x_j^2)^{q_j},\qquad
 H=\prod_{j=1}^3(2+x_j),\qquad d,q_j\in\ZZ_{\ge0}.
\end{equation}
For $0<a_i\le1$, write
\[
 K(a_1,a_2,a_3)=\{x_1+x_2+x_3=0:\ |x_i|\le a_i\}.
\]
The normalization of Lebesgue measure on this plane will not matter in the
next two sign lemmas.

\begin{lemma}[positive-root box inequality]\label{lem:positive-box}
For $d\ge1$ and $j\in\{1,2,3\}$,
\[
 \int_{K(a_1,a_2,a_3)}\frac{x_jH^d}{2+x_j}\,dx<0.
\]
\end{lemma}

\begin{proof}
By relabeling take $j=1$. Pair opposite polar rays of the centrally
symmetric convex body $K$. Normalize a nontrivial pair as
$tw=(at,bt,ct)$, $0\le t\le1$, where $a>0$, $a+b+c=0$, and
$|a|,|b|,|c|\le1$. Apart from a positive angular factor, its contribution is
\begin{equation}\label{eq:raypair}
 J=a\int_0^1t^2\{P(t)-P(-t)\}\,dt,\qquad
 P(t)=(2+at)^{d-1}(2+bt)^d(2+ct)^d.
\end{equation}
If $bc\le0$, then, for $L=\log(P(t)/P(-t))$,
\[
 L'=4\left(d\sum_{\xi\in\{a,b,c\}}
       \frac{\xi}{4-t^2\xi^2}-\frac{a}{4-t^2a^2}\right),
\]
and putting the first sum over a common denominator gives
\[
 \sum_{\xi\in\{a,b,c\}}\frac{\xi}{4-t^2\xi^2}
 =\frac{abc\,t^2\bigl(12-\tfrac12(a^2+b^2+c^2)t^2\bigr)}
        {\prod_{\xi\in\{a,b,c\}}(4-t^2\xi^2)}\le0.
\]
Thus $L'<0$, and $P(t)<P(-t)$ for $t>0$.

Suppose $bc>0$. Write $b=-u$, $c=-v$, $a=u+v$, and
$\omega=uv\le a^2/4$. Then $(2+bt)(2+ct)=4-2at+\omega t^2$, so
$P(t)=(2+at)^{d-1}(4-2at+\omega t^2)^d$ and
$\partial_\omega P(t)=d\,t^2H_\omega(t)^{d-1}$, where
\[
 H_\omega(t)=(2+at)(4-2at+\omega t^2)=8-2(a^2-\omega)t^2+a\omega t^3.
\]
Regarding \eqref{eq:raypair} as $J(a,\omega)$ and differentiating gives
\[
 \partial_\omega J
 =ad\int_0^1t^4\{H_\omega(t)^{d-1}-H_\omega(-t)^{d-1}\}\,dt\ge0.
\]
It is therefore enough to take $u=v=a/2$. With $q=at$,
$J(a,a^2/4)=a^{-2}K_d(a)$, where
\[
 K_d(a)=\int_0^a q^2\left[(2+q)^{d-1}(2-q/2)^{2d}
              -(2-q)^{d-1}(2+q/2)^{2d}\right]dq.
\]
Using $z=(2+q)/6$ in the first term and $z=(2-q)/6$ in the second,
$K_d(a)$ is a positive constant times the mass to the right of $1/3$
minus the mass to its left for
\[
 f(z)=(z-1/3)^2z^{d-1}(1-z)^{2d}.
\]
Put $r=z/(1-z)$. The involution $r\mapsto r'=1/(4r)$ maps the right
interval into the left one. Including $|dr'/dr|=1/(4r^2)$, the ratio of
the transported left density to the right density is
\[
 \mathcal Q(r)=
 \frac{4^{2d+2}r^{d+1}(1+r)^{3d+3}}{(4r+1)^{3d+3}}.
\]
Here $\mathcal Q(1/2)=1$ and
\[
 (\log\mathcal Q)'=
 \frac{(d+1)(2r-1)^2}{r(1+r)(4r+1)}>0\qquad(r>1/2).
\]
Thus the left mass is strictly larger, so $K_d(a)<0$. Every paired ray
with $a\ne0$ contributes negatively; the directions with $a=0$ contribute
zero and have angular measure zero. This proves the lemma.
\end{proof}

\begin{lemma}[negative-root box inequality]\label{lem:negative-box}
For $d\ge1$ and $j\in\{1,2,3\}$,
\[
 \int_{K(a_1,a_2,a_3)}x_jH^d(2+x_j)\,dx>0.
\]
\end{lemma}

\begin{proof}
Take $j=1$, put $b=a_2$, $c=a_3$, and pair the slices $x_1=s$ and
$x_1=-s$. For $0<s<\min(a_1,b+c)$, set
\[
 J_s=[-b,b]\cap[s-c,s+c],\qquad
 h(s)=\int_{J_s}(2+x)^d(2+s-x)^d\,dx.
\]
The paired contribution is positive precisely when
\begin{equation}\label{eq:conv-ratio}
\frac{h(s)}{h(-s)}<\left(\frac{2+s}{2-s}\right)^{d+1}.
\end{equation}
For $s\ge b+c$ both slices have measure zero and contribute nothing.
Reflection in the integral for $h(-s)$ puts both integrals on $J_s$.
On this interval define
\[
 B=(2-x)(2-s+x)=C+sx-x^2,\qquad C=2(2-s),
 \qquad I_m=\int_{J_s}B^m\,dx.
\]
Since $(2+x)(2+s-x)=B+4s$,
\begin{equation}\label{eq:conv-expand}
 \frac{h(s)}{h(-s)}
 =\sum_{k=0}^d\binom dk(4s)^k\frac{I_{d-k}}{I_d}.
\end{equation}
We claim that, for $m\ge1$,
\begin{equation}\label{eq:moment-recurrence}
 \frac{I_{m-1}}{I_m}\le\frac{m+1}{mC}.
\end{equation}
The vertex of $B$ is $s/2$. If $J_s$ contains it, write
$B=q-y^2$, where $q=(2-s/2)^2=C+s^2/4$. Integrating $(yB^m)'$ and
using that the endpoints straddle $y=0$ gives
\[
 (2m+1)I_m\ge2mqI_{m-1}.
\]
This implies \eqref{eq:moment-recurrence}, since
$2q(m+1)-C(2m+1)=C+(m+1)s^2/2>0$.

If $J_s$ misses the vertex, reflection $x\mapsto s-x$ interchanges $b$
and $c$, so assume $J_s$ lies to its left. Then
$J_s=[L,b]\subset[-b,b]$, and $B$ is increasing there. The quotient
$I_{m-1}/I_m$ is the $B^m$-weighted mean of $1/B$; deleting the low-$B$
prefix $[-b,L]$ cannot increase it. It remains to prove
\eqref{eq:moment-recurrence} on $[-b,b]$. Put $P=C-x^2$ and expand
$T=(m+1)I_m-mCI_{m-1}$ in even powers of $s$. For $n=m-2r\ge1$, the
coefficient of $s^{2r}$, apart from the positive factor $\binom m{2r}$,
is
\[
 \begin{split}
 &\int_{-b}^b x^{2r}P^{n-1}\bigl\{(2r+1)C-(m+1)x^2\bigr\}\,dx\\
 &\quad=\bigl[x^{2r+1}P^n\bigr]_{-b}^b
   +n\int_{-b}^b x^{2r+2}P^{n-1}\,dx>0.
 \end{split}
\]
If $n=0$, the final coefficient is plainly positive. Thus $T>0$, and
\eqref{eq:moment-recurrence} follows.

Iteration gives
\[
 \frac{I_{d-k}}{I_d}\le C^{-k}\frac{d+1}{d-k+1}.
\]
Substitution in \eqref{eq:conv-expand} yields
\[
 \frac{h(s)}{h(-s)}
 \le\sum_{k=0}^d\binom{d+1}{k}\left(\frac{4s}{C}\right)^k
 <\left(1+\frac{4s}{C}\right)^{d+1}
 =\left(\frac{2+s}{2-s}\right)^{d+1}.
\]
The omitted $k=d+1$ term makes the inequality strict. This proves
\eqref{eq:conv-ratio}, hence the lemma.
\end{proof}

The line-share inequality needs a sharper derivative form of the same
convolution comparison.

\begin{lemma}[box-share inequality]\label{lem:boxshare}
For $a,b,c\in(0,1]$ and an integer $d\ge0$, put
\[
 K_{b,c}(x)=\int_{[-b,b]\cap[-x-c,-x+c]}
       \bigl((2+u)(2-x-u)\bigr)^d\,du
\]
and
\[
 \mathcal D_d(a,b,c)=-\int_{-a}^a x(2+x)^dK_{b,c}'(x)\,dx.
\]
The function $K_{b,c}$ is Lipschitz (indeed, piecewise polynomial), and
$K_{b,c}'$ denotes its almost-everywhere derivative.
Then $\mathcal D_d(a,b,c)>0$ for $d\ge1$. For $d=0$ it is
nonnegative, and $\mathcal D_0(1,1,1)=1$.
\end{lemma}

\begin{proof}
Assume first $d\ge1$. Set $J(s)=K_{b,c}(-s)$ and pair $s$ with $-s$:
\begin{equation}\label{eq:boxshare-pair}
 \mathcal D_d(a,b,c)=\int_0^a sG(s)\,ds,\qquad
 G(s)=(2+s)^dJ'(-s)-(2-s)^dJ'(s).
\end{equation}
The integrand vanishes for $s\ge b+c$, so take $0<s<b+c$; interchange
$b,c$ so that $b\le c$. On
\[
 I=[-b,b]\cap[s-c,s+c]=[L,U]
\]
write
\[
 A(x)=(2+x)(2+s-x),\quad B(x)=(2-x)(2-s+x),\quad m=d-1,
\]
and $M_A=\int_I A^m$, $M_B=\int_I B^m$. If $s<c-b$, then
$I=[-b,b]$ and
\[
 J'(s)=d\int_I(2+x)A^m,\qquad
 J'(-s)=d\int_I(2-x)B^m.
\]
If $c-b<s<b+c$, then $L=s-c$, $U=b$, and
\[
 J'(s)=-A(L)^d+d\int_I(2+x)A^m,\qquad
 J'(-s)=B(L)^d+d\int_I(2-x)B^m.
\]
Since $A'=B'=s-2x$, integration by parts yields
\[
 d\int_I(2-x)B^m=\frac{d(4-s)}2M_B+\frac{[B^d]_L^U}{2},
 \quad
 d\int_I(2+x)A^m=\frac{d(4+s)}2M_A-\frac{[A^d]_L^U}{2}.
\]
The main term of $G(s)$ is therefore
\begin{equation}\label{eq:boxshare-main}
 \frac d2\bigl((2+s)^d(4-s)M_B-(2-s)^d(4+s)M_A\bigr).
\end{equation}
The remaining boundary term in the central case is
\[
 \frac12\{(2+s)^d(B(U)^d-B(L)^d)
 +(2-s)^d(A(U)^d-A(L)^d)\}>0,
\]
because $U=b,L=-b$ and both differences of the bases equal $2bs$. In
the clipped case it is
\[
 \frac12\{(2+s)^d(B(L)^d+B(U)^d)
 +(2-s)^d(A(L)^d+A(U)^d)\}>0.
\]
It remains to prove, with $r=(2+s)/(2-s)$,
\begin{equation}\label{eq:boxshare-ratio}
 \frac{M_A}{M_B}<r^d\frac{4-s}{4+s}.
\end{equation}

First suppose that $I$ contains $h=s/2$. With $y=x-h$,
$B=q-y^2$, $q=(2-s/2)^2$, and $A=B+4s$. For
$N_j=\int_I B^j$, integration of $(yB^j)'$ gives
\[
 (2j+1)N_j-2jqN_{j-1}=[yB^j]_{\partial I}\ge0,
\]
where the endpoints straddle $y=0$. Iterating and expanding $(B+4s)^m$
gives
\[
 \frac{M_A}{M_B}\le
 S_m(z):=\sum_{k=0}^m\binom{m+\tfrac12}{k}z^k,
 \qquad z=\frac{4s}{q}.
\]
This is the degree-$m$ Taylor polynomial at zero of
$(1+z)^{m+1/2}$, whose Lagrange remainder is positive. Hence, with
$w=(4+s)/(4-s)$ and $1+z=w^2$,
\[
 \frac{M_A}{M_B}<w^{2m+1}<\frac{r^{m+1}}{w}.
\]
The last inequality follows from $w^2<r$; clearing denominators leaves
the positive difference $2s^3$. Since $d=m+1$, this is
\eqref{eq:boxshare-ratio}.

Suppose next that the central interval $[-b,b]$ misses $h$, so $b<h$.
Pair $x$ and $-x$ for $0\le x\le b$. Put
\[
 D_0=4-2s-x^2,\quad z_0=sx,\quad
 p=\frac{D_0-z_0}{D_0+z_0},\quad P=1-p,
\]
and
\[
 \theta=\frac{s-x}{2+s},\qquad \kappa=\frac{2-x}{2+s}.
\]
The exact identities
\[
 \frac{A(x)}{rB(x)}=1-\theta P,\qquad
 \frac{A(-x)}{rB(x)}=1-\kappa P
\]
hold. Here $0<p\le1$, $P=1-p$, and $0\le\theta,\kappa\le1$. Convexity
gives the chord bound
\[
 (1-\lambda P)^m\le1-\lambda+\lambda p^m
 \qquad(0\le\lambda\le1),
\]
with $m=0$ immediate. Applying this to $\theta$ and $\kappa$, using
$\theta+\kappa=1-2x/(2+s)\le1$, and dividing by $1+p^m$ gives
\[
 \frac{A(x)^m+A(-x)^m}{B(x)^m+B(-x)^m}
 \le r^m\left(1+\frac{2x}{2+s}\right)
 \le r^m\left(1+\frac{s}{2+s}\right)
 <r^{m+1}\frac{4-s}{4+s};
\]
the final gap is $(s+6)s^2/((2-s)(2+s)(4+s))>0$. Integration proves
\eqref{eq:boxshare-ratio}. Finally, if a clipped interval misses $h$,
then it lies to the left of $h$ and is a suffix of $[-b,b]$. There $B$
increases while $(A/B)^m=(1+4s/B)^m$ decreases. Deleting the low-$B$
prefix decreases the $B^m$-weighted average $M_A/M_B$, reducing to the
central case. Thus \eqref{eq:boxshare-ratio} always holds. The main term
\eqref{eq:boxshare-main} and the positive boundary term give $G(s)>0$,
and \eqref{eq:boxshare-pair} proves the assertion for $d\ge1$.

For $d=0$, $K_{b,c}$ is the convolution of two centered interval
indicators, hence is even and nonincreasing on $[0,\infty)$. Thus
$-xK_{b,c}'(x)\ge0$. If $a=b=c=1$, then
$K_{1,1}(x)=2-|x|$ on $[-1,1]$, and $\mathcal D_0(1,1,1)=1$.
\end{proof}

\begin{proof}[Proof of Theorem~\ref{thm:zerosum}]
We apply the three box inequalities to \eqref{eq:zerosum-normal}. For
$q=0$ let $\mu_q=\delta_1$, while for $q\ge1$ put
\begin{equation}\label{eq:layer-measure}
 d\mu_q(a)=3^q\,d\delta_1(a)
       +2qa(4-a^2)^{q-1}\,da\qquad(0<a<1).
\end{equation}
Then, exactly on $|u|\le1$,
\begin{equation}\label{eq:layer-cake}
 (4-u^2)^q
 =\int_{(0,1]}\mathbf1_{\{|u|\le a\}}\,d\mu_q(a).
\end{equation}
Thus every product of pair factors is a positive mixture of indicators of
the centered boxes $K(a_1,a_2,a_3)$.

Assume first $d\ge1$. A positive-root occurrence on line $j$ has
singleton integral $\int_QF x_j/(2+x_j)$, which is negative by
\eqref{eq:layer-cake} and Lemma~\ref{lem:positive-box}. A negative-root
occurrence requires $q_j\ge1$. Peeling off one factor
$4-x_j^2=(2+x_j)(2-x_j)$ rewrites its singleton as
\[
 -\int_Qx_jH^d(2+x_j)(4-x_j^2)^{q_j-1}
       \prod_{i\ne j}(4-x_i^2)^{q_i},
\]
which is negative by Lemma~\ref{lem:negative-box}. If $d=0$, then $F$ is
centrally even, and pairing $x$ with $-x$ gives, for either occurring sign,
\[
 \varphi_{\pm\gamma_j}=-\int_QF\frac{x_j^2}{4-x_j^2}<0.
\]
This proves condition (a) of Proposition~\ref{prop:reduction}.

It remains to prove the line-share inequalities. Fix the first line. On
$[-1,1]$ set
\[
 w_j(u)=(2+u)^d(4-u^2)^{q_j},\qquad
 g_j(u)=w_j(u)\mathbf1_{\{|u|\le1\}},
\]
and extend only $g_j$ by zero to $\mathbb R$. Put
\[
K(x)=\int_{\mathbb R}g_2(y)g_3(-x-y)\,dy,
 \qquad I=\int_{-1}^1w_1(x)K(x)\,dx=\int_QF.
\]
The two edge integrals on this line sum to
\[
 B_1=w_1(-1)K(-1)+w_1(1)K(1)
 =\int_{E_{\gamma_1}}F+\int_{E_{-\gamma_1}}F.
\]
Since the multiplicities on the line are
$d+q_1$ copies of $\gamma_1$ and $q_1$ copies of $-\gamma_1$,
\[
\sum_{t_i=\pm\gamma_1}\varphi_i
 =\int_{-1}^1xw_1K\left(\frac{d+q_1}{2+x}-\frac{q_1}{2-x}\right)dx
 =\int_{-1}^1xw_1'K\,dx.
\]
By condition (a), the share $S_1:=S_{\gamma_1}$ is
$B_1-\sum_{t_i=\pm\gamma_1}\varphi_i$. Since $K$ is absolutely
continuous, integration by parts on $[-1,1]$ gives
\[
 B_1=I+\int_{-1}^1x(w_1'K+w_1K')\,dx,\qquad
 I-S_1=-\int_{-1}^1xw_1K'\,dx.                  \tag{*}
\]
The layer kernels are uniformly compactly supported and Lipschitz, so
substitution of \eqref{eq:layer-cake} in (*) and Fubini's theorem, applied
also to their almost-everywhere derivatives, give
\[
 I-S_1=\int_{(0,1]^3}\mathcal D_d(a,b,c)\,
 d\mu_{q_1}(a)d\mu_{q_2}(b)d\mu_{q_3}(c).
\]
For $d\ge1$ this is strictly positive by Lemma~\ref{lem:boxshare}. For
$d=0$ every layer is nonnegative, while the atom $(a,b,c)=(1,1,1)$ has
positive mass and contributes $1$. Thus $S_1<I$ in all cases. The same
argument on the other two root lines proves condition (b), so
Proposition~\ref{prop:reduction} gives stability.

Finally, for $k=n-2\ge2$, use $k/2$ antipodal pairs if $k$ is even, and
one rotation triple together with $(k-3)/2$ antipodal pairs if $k$ is odd.
These zero-sum multisets give the asserted examples in every dimension.
\end{proof}

\begin{lemma}[extreme-exponent ordering]\label{lem:moment-order}
Assume $d\ge1$ in \eqref{eq:zerosum-normal}, and put
\[
 M_i=\int_Qx_iF(x)\,dx\qquad(i=1,2,3).
\]
If $q_i=\max_\nu q_\nu$, $q_j=\min_\nu q_\nu$, and $q_i>q_j$, then
$M_i<M_j$. In particular, the hexagon moment can vanish only if
$q_1=q_2=q_3$.
\end{lemma}

\begin{proof}
Write
\[
 E=8-\sum_{\nu=1}^3x_\nu^2,\qquad e=x_1x_2x_3,
 \qquad G=\prod_{\nu=1}^3(4-x_\nu^2)^{q_\nu},
\]
so that $H=E+e$ and $H(-x)=E-e$. The divided difference
\[
 R_d=\sum_{m=0}^{\lfloor(d-1)/2\rfloor}
 \binom d{2m+1}E^{d-2m-1}e^{2m}
 =\frac{(E+e)^d-(E-e)^d}{2e}
\]
where the quotient at $e=0$ is interpreted as $dE^{d-1}$, is continuous
and strictly positive on $Q$. It is invariant under coordinate permutations
and central inversion. Since $G$ is centrally even, central symmetry gives
\begin{equation}\label{eq:moment-central}
 M_i=\frac12\int_Qx_i\bigl(H(x)^d-H(-x)^d\bigr)G(x)\,dx
 =\int_Qx_ieR_dG\,dx.
\end{equation}

Up to a null set, $Q$ is the disjoint union of the twelve images, under
coordinate permutations and central inversion, of
\[
 \mathcal T=\{(a,b):0<a<b,\ a+b<1\},\qquad
 (x_1,x_2,x_3)=(a,b,-c),\quad c=a+b.
\]
These maps have Jacobian of absolute value one in lattice coordinates.
Put
\[
 A=4-a^2>B=4-b^2>C=4-c^2>0,
\]
and attach the scores $s_A=-a$, $s_B=-b$, $s_C=c$ to the corresponding
even factors $U_A=A$, $U_B=B$, $U_C=C$. Define
\[
 V_i=\sum_{\sigma:\{1,2,3\}\overset{\sim}{\longrightarrow}\{A,B,C\}}
 s_{\sigma(i)}\prod_{\nu=1}^3U_{\sigma(\nu)}^{q_\nu}.
\]
At $(a,b,-c)$ one has $e=-abc$, and assigning the magnitude labelled
$L$ to coordinate $i$ gives $x_ie=abc\,s_L$. The negative sector gives
the same product. Hence \eqref{eq:moment-central} becomes
\begin{equation}\label{eq:moment-sector}
 M_i=2\int_{\mathcal T}abc\,R_d(a,b,-c)V_i(a,b)\,da\,db.
\end{equation}

Let $k$ be the remaining index and write
\[
 q_i=q+r,\qquad q_j=q,\qquad q_k=q+h,
 \qquad r>0,\quad 0\le h\le r.
\]
Pairing in $V_i-V_j$ each assignment with the one that interchanges the
labels at $i$ and $j$, and grouping by the label at $k$, gives
\begin{align*}
 V_i-V_j=(ABC)^q\bigl[{}
 &(b-a)(A^r-B^r)C^h\\
 &-(2a+b)(A^r-C^r)B^h\\
 &-(a+2b)(B^r-C^r)A^h\bigr].
\end{align*}
The absolute value of the second term alone is strictly larger than the
first, since
\[
 2a+b>b-a,\qquad A^r-C^r>A^r-B^r,\qquad B^h\ge C^h,
\]
and the third term is also strictly negative. Thus $V_i<V_j$ throughout
$\mathcal T$. Equation~\eqref{eq:moment-sector} and $R_d>0$ imply
$M_i<M_j$.
\end{proof}

\begin{proof}[Proof of Theorem~\ref{thm:roottwist-ke}]
Use the normal form \eqref{eq:zerosum-normal}, after global negation if
necessary; this preserves moment vanishing and both symmetry alternatives.
If $d=0$, the multiplicities of every root and its negative
agree. Thus the multiset is negation-invariant, $F$ is centrally even,
and its hexagon moment vanishes.

Suppose $d\ge1$. If $q_1=q_2=q_3$, then the multiplicities $d+q_j$ of
the $\gamma_j$ and the multiplicities $q_j$ of the $-\gamma_j$ are
separately constant in $j$. Hence the multiset and $F$ are invariant under
$\tau$. The three numbers $M_j=\int_Qx_jF$ are equal; since
$x_1+x_2+x_3=0$, they vanish. Conversely, if the hexagon moment vanishes,
then every $M_j$ vanishes, and Lemma~\ref{lem:moment-order} forces
$q_1=q_2=q_3$. For $d\ge1$ the multiset cannot be negation-invariant,
because the multiplicities of $\gamma_j$ and $-\gamma_j$ differ by $d$.

The moment therefore vanishes exactly in the two stated symmetry cases.
Lemma~\ref{lem:prism}(3) identifies moment vanishing with the
K\"ahler--Einstein condition.
\end{proof}

The general result contains two of the following three explicit families;
the third shows that zero sum is not necessary.

\begin{theorem}\label{thm:families}
For every $r \ge 1$, $m \ge 1$, and $s\ge0$, the root-twist varieties
\[
X\bigl((\alpha, -\alpha)^{\times r}\bigr) \quad (n = 2r + 2)
\qquad\text{and}\qquad
X\bigl((\alpha, \beta, -(\alpha{+}\beta))^{\times m}\bigr) \quad (n = 3m+2)
\]
have $(-K)$-stable tangent bundle and are K\"ahler--Einstein. The varieties
\[
Y_s=X\bigl((\alpha,-\alpha)^{\times s},\alpha,\beta\bigr)
\qquad (n=2s+4)
\]
also have stable tangent bundle but are not K\"ahler--Einstein. All three
families are smooth toric Fanos of Picard rank $n+2$. In particular, stable
K\"ahler--Einstein examples of rank $n+2$ exist in every even dimension
$n\ge4$ and every dimension $n\equiv2\pmod3$, $n\ge5$, while stable
non-K\"ahler--Einstein examples of that rank exist in every even dimension
$n\ge4$.
\end{theorem}

\begin{proof}
\emph{Smooth Fano} is Lemma~\ref{lem:prism}(1).

\emph{K\"ahler--Einstein.} The pair and rotation-orbit multisets satisfy
the two symmetry alternatives of Theorem~\ref{thm:roottwist-ke}.

\emph{Stability.} The pair and triple twist multisets are nonzero and
zero-sum, so their tangent bundles are stable by
Theorem~\ref{thm:zerosum}.

\emph{The unbalanced family.} Put $x=\ell_\alpha$, $y=\ell_\beta$,
$q(u)=4-u^2$, and
$M_j=\int_0^1u^jq(u)^s\,du$. For $Y_s$ one has
$F=q(x)^s(2+x)(2+y)$. For $0\le u\le1$, the slices of $Q$ at $x=u$ and
$x=-u$ are respectively $[-1,1-u]$ and $[-1+u,1]$ in the $y$-coordinate,
and
\[
\int_{-1}^{1-u}(2+y)\,dy=4-3u+\tfrac12u^2,
\qquad
\int_{-1+u}^{1}(2+y)\,dy=4-u-\tfrac12u^2.
\]
Consequently, with $I_0=\int_QF$,
\[
I_0=16M_0-8M_1-2M_2+M_3,
\qquad
\varphi_\alpha=\varphi_\beta=-2M_2+M_3
=-\int_0^1u^2(2-u)q(u)^s\,du<0.
\]
If $s\ge1$, the remaining singleton integral is
\[
\varphi_{-\alpha}=-\int_0^1u^2q(u)^{s-1}
(24-12u-2u^2+u^3)\,du<0,
\]
since the last factor is at least $11$ on $[0,1]$. Thus
Proposition~\ref{prop:reduction}(a) holds (and for $s=0$ there is no
$-\alpha$ twist).

Let $B_\rho=\int_{E_\rho}F+\int_{E_{-\rho}}F$. Direct edge integration
gives
\[
B_\beta=8M_0-2M_1,
\qquad
B_{\alpha+\beta}=8M_0+2M_1-2M_2.
\]
There is one $\beta$ twist and no $\pm(\alpha+\beta)$ twist, hence
\[
S_\beta=8M_0-2M_1+2M_2-M_3,
\qquad
S_{\alpha+\beta}=8M_0+2M_1-2M_2.
\]
It follows that
\[
I_0-S_\beta=2\int_0^1q(u)^s(1-u)(4+u-u^2)\,du>0.
\]
The functions $q(u)^s$ and $h(u)=8-10u+u^3$ are nonincreasing on
$[0,1]$. Chebyshev's integral inequality therefore gives
\[
I_0-S_{\alpha+\beta}=\int_0^1q(u)^sh(u)\,du
\ge M_0\int_0^1h(u)\,du=\tfrac{13}{4}M_0>0.
\]
Finally the share identity yields
\[
I_0-S_\alpha=S_\beta+S_{\alpha+\beta}-I_0
=2\int_0^1q(u)^s(4u+u^2-u^3)\,du>0.
\]
Proposition~\ref{prop:reduction}(b) now proves stability.

On the other hand, central symmetry of $Q$ gives
\[
\int_QxF=2\int_Qq(x)^s(x^2+xy)
=2\int_0^1u^2(2-u)q(u)^s\,du>0.
\]
The hexagon moment is therefore nonzero, so Lemma~\ref{lem:prism}(3) shows
that $Y_s$ is not K\"ahler--Einstein.
\end{proof}

The first unbalanced member $Y_0=X(\alpha,\beta)$ is the census entry
\textsf{F.4D.0061}: it has $(-K)^4=278$, $\mu(T_X)=139/2$, and maximal
subsheaf slope $206/3$, with barycenter
$(10,10,-5,-5)/139$ in the construction coordinates. Explicitly, writing
rays as row vectors, right multiplication by the unimodular matrix
\[
\begin{pmatrix}
0&0&-1&0\\ -1&0&1&0\\ 0&0&0&1\\ 0&1&0&0
\end{pmatrix}
\]
sends its ten rays onto the stored rays of \textsf{F.4D.0061}. Thus the failure
of twist-sum necessity is already visible in dimension four, and
Theorem~\ref{thm:families} propagates it to every even dimension.

\begin{remark}\label{rem:margins}
The margins are explicit: for the triple family every share sits uniformly
$\tfrac13 \int_Q F$ below critical, while for the pair family the margin on
the twisted line is $2J = (4^{r+1} - 3^{r+1})/(r+1) = o(\int_Q F)$, where
$J=\int_0^1u(4-u^2)^r\,du$: the
one-line family is asymptotically critical on its own line, consistent
with the near-critical maximal slopes these varieties exhibit in the
census. (For the asymptotic claim, the numerator is $O(4^r/r)$, while
slicing over $|\ell_\alpha|\le r^{-1/2}$ gives
$\int_QF\gg4^r/\sqrt r$.)

The normal form \eqref{eq:zerosum-normal} and the positive decomposition
\eqref{eq:layer-cake} show why the mixed cases remain tractable despite their
lack of symmetry: the relevant inequalities are preserved under positive
averaging on centered-box layers; when $d=0$, the full-box atom supplies the
required strictness. Thus zero sum is sufficient
for stability in all dimensions, while the unbalanced family shows that it is
not necessary.
\end{remark}

By Lemma~\ref{lem:prism}(3), the K\"ahler--Einstein question is the moment
problem $\int_QxF(x)\,dx=0$. Theorem~\ref{thm:roottwist-ke} solves it
under the zero-sum hypothesis; classifying vanishing moments for arbitrary
root multisets remains open.

\section{Structural observations and open questions}\label{sec:structure}

\subsection{Antipodal directions in destabilizers}\label{sec:antipodal}
Steffens proved that a Fano three-fold whose tangent bundle is not
stable admits a destabilizing subsheaf that is the relative tangent
sheaf of the contraction of an extremal face \cite{Steffens1996}.
Peternell \cite[Conjecture~3.21]{Peternell2001} conjectured the
analogous statement in all dimensions, with a contraction
$f\colon X\to Y$ satisfying $\dim Y<\dim X$.
Kanemitsu's Picard-rank-one examples \cite{Kanemitsu2021} disprove the
conjecture in that rank. Reynolds \cite[Theorem~2.2.6]{Reynolds2023}
disproves it in Picard rank four and dimension four: Batyrev--Sato
no.~$82$ has unstable tangent bundle, but no Mori fibration has
destabilizing relative tangent sheaf. Here a Mori fibration may contract
an extremal face of any dimension, not just an extremal ray.
Reynolds also gives the full Harder--Narasimhan filtrations of all $74$
unstable four-folds and a destabilizing Mori fibration for each of the
other $73$ \cite[Appendix~A]{Reynolds2023}.

For every unstable case our computation forms the sum of all
maximal-slope ray-spanned subspaces,
which is the canonical first Harder--Narasimhan subspace. Its dimension is
one in $5/7$, $38/74$, $375/675$, $3865/6620$ of the unstable cases in
dimensions $3, 4, 5, 6$. It contains an antipodal pair $\pm u$ of rays in
$7/7$, $68/74$, $558/675$, $4877/6620$ of them. Such a pair is the
Klyachko datum of a rank-one relative tangent candidate; determining whether
it comes from an equivariant $\PP^1$-fibration requires checking compatibility
with the fan. The no-pair cases (six in dimension four, $117$ in dimension
five, $1743$ in dimension six) include HN subspaces spanned by a single ray
($1$, $50$, $986$ in dimensions $4, 5, 6$).

In Reynolds's counterexample, \textsf{F.4D.0002} in our numbering, the
first Harder--Narasimhan term has rank $3$, degree $278$, and slope
$278/3>92=\mu(T_X)$. Its subspace contains an antipodal pair, so this
condition does not imply a destabilizing Mori fibration. Comparing the
recorded subspaces with Mori fibrations in dimensions five and six would
extend Reynolds's analysis.

\subsection{High Picard rank in the Fano census}
Question 1.5 of \cite{HNS2022} asks whether, for fixed $(n, \rho)$, only
finitely many smooth projective toric $n$-folds of Picard rank $\rho$
have a polarization making $T_X$ stable. The present census does not
address that question: there are only finitely many smooth toric Fano
$n$-folds for each $n$, so finiteness among anticanonical Fanos is
immediate. What the tables of Section~\ref{sec:results} do show is the
shape of the stable locus at $-K$: it thins rapidly with $\rho$, and
its high-$\rho$ boundary is occupied by a single structured example in
dimension five (Section~\ref{sec:roottwist}). The root-twist varieties have
$\rho=n+2$, and Theorem~\ref{thm:zerosum} gives stable examples at that rank
in every dimension $n\ge4$; whether they exhaust the high-rank stable locus
is open.

\subsection{Open questions}
\begin{enumerate}
\item Classify arbitrary root multisets, without the zero-sum hypothesis,
whose hexagon moment vanishes. More generally, characterize the toric
fibrations for which a twist
turns all slope equalities of the fiber--base decomposition strict.
\item Determine the maximal Picard rank of a stable example in each
dimension. Is it $n + 2$, attained by the root-twist family, for all
$n \ge 5$? By Theorem~\ref{thm:zerosum}, rank $n + 2$ is attained in
every dimension $n\ge4$; only maximality is open.
\item Extend the census from $-K$ to the full nef cone (the stability
chambers of \cite{HNS2022}, Ex.~5.2), where already in dimension three
the chamber structure is not polyhedral.
\item Quantify the asymptotics of the stable, polystable and
K\"ahler--Einstein fractions as $n \to \infty$.
\end{enumerate}

\providecommand{\MR}{\relax\ifhmode\unskip\space\fi MR }
\providecommand{\MRhref}[2]{%
  \href{http://www.ams.org/mathscinet-getitem?mr=#1}{#2}
}
\providecommand{\href}[2]{#2}

\bigskip
\noindent{\sc Bernd Johannes Wuebben, New York, NY,} \texttt{wuebben@gmail.com}

\end{document}